\documentclass[a4paper,11pt,leqno]{article}
 \usepackage{tikz}
 \usepackage{framed}
 \usepackage{natbib}
\setcitestyle{authoryear,open={(},close={)}}
\usepackage{enumitem}
\usepackage{lingmacros}
 \usepackage{url}
\usepackage{geometry}
\usepackage{amsmath}
\usepackage{amsfonts}
\usepackage{amssymb}
\usepackage{amsthm}
\usepackage{qtree}
\usepackage{graphicx}
\usepackage{leftidx}

\usepackage{bussproofs}
\usepackage{libertinus}
\usepackage{libertinust1math}
\usepackage[T1]{fontenc}
\usepackage[applemac]{inputenc} 
\usepackage{scalerel}
\usepackage{setspace}
\usepackage{hyperref}
\usepackage{verbatim}
\newtheorem{de}{Definition}

\newtheorem{prop}[de]{Proposition}
\newtheorem{lem}[de]{Lemma}
\newtheorem{cor}[de]{Corollary}
\newtheorem{rem}[de]{Remark}

\newtheorem{exam}[de]{Example}

\usepackage{todonotes}
\usepackage{verbatim}

		\newcommand{\Llb}[1]{\LeftLabel{{\sc (#1)}}}

\newcommand{\subdot}[1]{\oalign{$#1$\cr\hfil.\hfil}}

\newcommand\T{\scaleobj{0.9}{\mathrm{T}}}

\newcommand\lt{\mathcal{L}_{\rm T}}
\newcommand\lc{\mathcal{L}^{\rhd}}
\newcommand\lct{\mathcal{L}^{\rhd}_{\T}}

\newcommand{\restr}{\!\upharpoonright\!}

\newbox\gnBoxA
\newdimen\gnCornerHgt
\setbox\gnBoxA=\hbox{$\ulcorner$}
\global\gnCornerHgt=\ht\gnBoxA
\newdimen\gnArgHgt

\def\gn #1{%
\setbox\gnBoxA=\hbox{$#1$}%
\gnArgHgt=\ht\gnBoxA%
\ifnum     \gnArgHgt<\gnCornerHgt \gnArgHgt=0pt%
\else \advance \gnArgHgt by -\gnCornerHgt%
\fi \raise\gnArgHgt\hbox{$\ulcorner$} \box\gnBoxA %
\raise\gnArgHgt\hbox{$\urcorner$}}

\allowdisplaybreaks

\usepackage{eso-pic,picture,graphicx} 

	\makeatletter
	\newcommand{\draftversion}{
		\AddToShipoutPicture{%
			\setlength{\@tempdimb}{.5\paperwidth}%
			\setlength{\@tempdimc}{.5\paperheight}%
			\setlength{\unitlength}{1pt}%
			\protect{\put(\strip@pt\@tempdimb,\strip@pt\@tempdimc){%
					\makebox(-515,0){\rotatebox{90}{\textcolor[gray]{0.6}%
							{
								Please ask before citing -- email \href{mailto:johannes.stern@bristol.ac.uk}{Johannes} 
						}}}
					\makebox(515,0){\rotatebox{270}{\textcolor[gray]{0.6}%
							{
								{\Huge  Draft} \today 
							}  }}}}}
	}
	\makeatother

\usetikzlibrary{positioning,arrows,calc}
\tikzset{
modal/.style={>=stealth?,shorten >=1pt,shorten <=1pt,auto,node distance=1.5cm,
semithick},
world/.style={circle,draw,minimum size=0.5cm,fill=gray!15},
point/.style={fill=black,circle,inner sep=1pt},
	arrow/.style={-{Latex[width=1mm]},shorten <=2pt,shorten >=2pt},
reflexive above/.style={->,loop,looseness=7,in=120,out=60},
reflexive below/.style={->,loop,looseness=7,in=240,out=300},
reflexive left/.style={->,loop,looseness=7,in=150,out=210},
reflexive right/.style={->,loop,looseness=7,in=30,out=330}
}

\DeclareSymbolFont{symbolsC}{U}{txsyc}{m}{n}
\DeclareMathSymbol{\strictif}{\mathrel}{symbolsC}{74}
\DeclareSymbolFont{symbolsC}{U}{txsyc}{m}{n}
\DeclareMathSymbol{\boxright}{\mathrel}{symbolsC}{128}

\begin{document}
\title{Adequate conditionals and Kripke's theory of truth}
\author{Johannes Stern\\Department of Philosophy\\University of Bristol\\johannes.stern@bristol.ac.uk}
\date{}
\maketitle

\begin{abstract}In this paper we show how to introduce a conditional  to Kripke's theory of truth that respects the deduction theorem for the consequence relation associated with the theory. To this effect we develop a novel supervaluational framework, called strong Kleene supervaluation, that we take to be a promising framework for handling the truth-conditions of non-monotone notion in the presence of semantic indeterminacy more generally.
\end{abstract}

\section{Introduction}
After almost 50 years of Kripke's {\it Outline of a Theory of Truth} Kripke's inductive definition of a self-applicable truth predicate remains one of the most prominent semantic approaches to truth. Of course, a lot of work has been carried out since: Gupta, Herzberger and Belnap have given us the Revision theory of truth \citep{her82a, gup82, beg93} while others have extended Kripke's proposal to tackle the revenge problem \citep[cf., e.g.,][]{gla04,sim18,schl10}; while still others have applied it to intensional languages \citep[cf., e.g.,][]{haw09,ste14b,ste15a,ste21}. An important limitation of Kripke's proposal is that it is difficult to introduce conditionals or restricted quantifiers such as ${\sf Every}$ or ${\sf Most}$ to Kripke's theory of truth, that is, to formalize forms of conditional reasoning within the Kripkean framework. The problem is due to the fact that the truth conditions of such notions are non-monotone and this seems to be non-accidental: changing the truth-conditions to restore monotonicity seems to trigger a loss of central characteristic features of these notions.\footnote{By monotonicity we refer to the idea that an increase in semantic information leads to an increase of sentences receiving a classical truth value.} There have been important contributions to this problem and in particular the problem of introducing conditionals to Kripke's theory of truth, and, more generally, theories of naive truth by, e.g., \cite{yab03,bea09,bac13,ros16,iaro24} and, notably, \cite{fie08,fie16,fie21}, but all these proposals sacrifice important logical properties of the conditional in favor of the so-called {\it transparency} of the truth predicate.\footnote{We refer to Section \ref{cktt} for a precise introduction of this terminology.} The aim of this paper is to develop a general strategy for simultaneously introducing different conditionals, restricted quantifiers and, more generally, non-monotone notions that, in contrast, preserves the salient logical properties but sacrifices the transparency of the truth predicate. For the purpose of this paper we focus on conditionals, although our strategy applies more widely. We show how to construct Kripkean truth models for languages that include logically adequate conditionals. We shall argue that this requires the conditional to respect the deduction theorem for the consequence relation associated with the truth theory. 

To this effect we develop a supervaluational semantics, which in contrast to orthodox supervaluational semantics, quantifies over partial precisification instead of classical precisification. On this semantics a conditional is true in a given partial model $M$, if it is true in all available partial models that are semantically compatible with $M$. The crux of the matter is then to specify the notions of available partial model and semantic compatibility. This leads us to the framework of {\it strong Kleene supervaluations} and to {\it strong Kleene supervaluation structures}. A strong Kleene supervaluation structure consists of a non-empty domain, a set of strong Kleene interpretations and an admissibility relation on these interpretations: a conditional sentence will be true in a strong Kleene structure relative to an interpretation $J$ iff it is true at all admissible interpretations $J'$ extending $J$.  Our strategy for handling conditionals shares important aspects with Kripke semantics for intuitionistic logic and, in particular, the truth conditions of  the intuitionistic conditional. Ultimately, our approach amounts to introducing an intuitionistic conditional to strong Kleene logic and strong Kleene supervaluation for a conditional, as we introduce it in this paper, turns out to be notational variant of constant domain Kripke semantics for ${\sf N3}$-Nelson logic as introduced by \cite{tho69}.\footnote{We refer to \cite{wan91,wan01} for an introduction and discussion of Nelson logic.}

The main task of the paper is to show how to construct strong Kleene supervaluation structures for a language with a self-applicable truth predicate that comprises interpretations at which $\varphi$ and $\T\gn\varphi$ receive the same truth value for every sentence $\varphi$. Establishing the existence of such interpretations shows that we can apply Kripke's theory of truth to non-classical logics that have an adequate conditional: by constructing a Kripkean truth model for a language with a conditional we show how to introduce an ${\sf N3}$-conditional to Kripke's theory of truth. Incidentally, our construction thus answers a question posed by Albert Visser who asks whether we can ``{\it extend (\dots) Kripke's solution} [Kripke's theory of truth, JS] {\it with, say, Thomasonian implication?}'' \citep[p.~645]{vis89}. However, we do not stop there. We also show how strong Kleene supervaluation can be combined with ordering semantic to provide adequate truth conditions for indicative and subjunctive conditionals in terms of strict or variably strict implication, that is, our strategy can be employed to give truth conditions for various natural language conditionals in the context of Kripke's theory of truth. Whilst in this paper we focus on various conditionals we think of strong Kleene supervaluation as a general strategy and framework for handling non-monotone notions such as restricted quantifiers in the presence of semantic indeterminacy broadly understood.

\paragraph{Plan of the paper}
We start our investigation by a brief discussion of Kripke's theory of truth and the challenge to Kripke's theory of truth posed by conditionals and conditional reasoning more generally (Section \ref{cktt}). In Section \ref{CDT} we discuss the dilemma Curry's paradox poses to truth theorist: either truth is not transparent or the conditional does not respect the deduction theorem. We take the latter to imply that the truth theory would lack and adequate conditional, which we take to be a non-negiotable requirement on any satisfactory truth theory. We then turn to ideas motivating the framework of strong Kleene supervaluation which we introduce in detail in Section \ref{sks}. In Section \ref{TS} we expand supervaluation structures for arbitrary first-order languages to provide an interpretation of the language extended by a self-applicable truth predicate. Section \ref{KTS} shows how to single out so-called Kripkean truth structures amongst the strong Kleene supervaluation structures of the language with a self-applicable truth predicate. Kripkean truth structures are truth structures that contain interpretations at which $\varphi$ and $\T\gn\varphi$ will have the same truth value for every sentence $\varphi$ of the language with the self-applicable truth predicate. In Section \ref{truththeo} we examine the properties of the truth predicate and the conditional in Kripkean truth structures, whilst Section \ref{complex} discusses the recursion-theoretic complexity of the construction. Up to this point our semantics was concerned with a logically adequate conditional for a non-classical logic, that is, a non-classical conditional that assumes the role material implication has for classical logic in the non-classical logic under consideration. In the final section of the paper, Section \ref{msks} we show how the framework can be extended to a semantics for indicative and/or subjunctive conditionals. This amounts to combining strong Kleene supervaluation with ordering semantics  and to show how the construction of Kripkean truth structures can be carried out in the extended framework.




\section{Conditionals and Kripke's Theory of Truth}\label{cktt}
Kripke's theory of truth \citep{kri75} seeks to define the interpretation of an object-linguistic, self-applicable truth predicate. Of course, due to \cite{tar35} we know that no classical model can satisfy the $\T$-scheme, that is, in no classical model is it possible that $\varphi$ and $\T\gn\varphi$ have the same semantic value for every sentence $\varphi$.\footnote{That is, modulo the type of construction given by \cite{gup82}.} Kripke's insight was that the latter (but not the former) can be achieved if we resort to non-classical models: relative to a non-classical models $\mathcal{M}$ one can find an interpretation of the truth predicate $S$ such that for all sentences $\varphi$ 
\begin{flalign*}&\tag{{\sc Naivity}}\label{fp}&&(M,S)\Vdash\T\gn\varphi\text{ iff }(M,S)\Vdash\varphi.&&\end{flalign*}
Models for which {\sc Naivity} holds will be called models of \emph{naive} truth. Throughout this paper it is important to distinguish {\sc Naivity} from {\sc Transparency}. Transparent truth models are  models of naive truth that satisfy the following stronger property:
\begin{flalign*}&\tag{{\sc Transparency}}\label{sfp}&&(M,S)\Vdash\psi(\T\gn\varphi/\chi)\text{ iff }(M,S)\Vdash\psi(\varphi/\chi).&&\end{flalign*}
Roughly put {\sc Transparency} is {\sc Naivity} plus  $M$-extensionality, that is, the assumption that the truth value of a sentence in the model  $(M,S)$ only depends on the truth values of its subsentences in $(M,S)$. As we discuss below, {\sc Transparency} holds for some but not all variants of Kripke's theory of truth.

The standard strategy of showing the existence of naive (transparent) truth models is by defining a suitable monotone operation on $\mathcal{P}({\sf Sent})$ relative to some ground model $M$. Such an operation is also called a {\it Kripke jump}. There is some flexibility as to which logic and evaluation scheme is employed, but the most prominent variant of the Kripke jump is the operation
$$\mathcal{K}_{3}(S)=\{\varphi\,|\,(M,S)\Vdash_{\sf K3}\varphi\},$$ 
that is, the operation $\mathcal{K}_{3}$ applied to a set of sentences $S$ yields the set of sentences that are true in the model $(M,S)$ according to the strong Kleene scheme. A cardinality argument then establishes that the operation must have a fixed point, indeed, a minimal fixed point and if the operation has been suitably defined these fixed points yield attractive interpretations of truth predicate over $M$: if $S$ is fixed point, then $(M,S)$ is a naive truth model.\footnote{Monotonicity guarantees the existence of fixed points. It does not guarantee the existence of non-trivial fixed points. We will come back to this issue at a later stage.} Instead of using a simple cardinality argument the existence of fixed point can also be shown using an algebraic argument, that is, the Knaster-Tarski fixed-point theorem and variants thereof \citep[cf.][]{vis89}. This approach does not only establish the existence of a minimal fixed point, but also provides valuable information about the structure of the different fixed points of the operation.

One principal downside of Kripke's theory of truth is already implicit in the formulation of its positive result, that is, the existence of a naive truth model: the equivalence of $\varphi$ and $\T\gn\varphi$ relative to a model $(M,S)$ cannot be stated in the objectlanguage. There is no conditional that enables us to express the metalinguistic `iff' in the objectlanguage. Of course, this is no mere coincident as such a conditional would re-introduce the Liar paradox. More importantly, the partial logics/semantics employed in the context of Kripke's theory of truth do not comprise an attractive conditional that adequately captures conditional reasoning in language and logic. It is arguably for that reason that \cite[p.~95]{fef84} famously complained that in such logics ``{\it no sustained ordinary reasoning can be carried on}'' \citep{fef84}. A theory of truth without an adequate conditional, that is, a conditional that enables us to carry out conditional reasoning is left wanting. So the question arises how an adequate conditional can be integrated with Kripke's theory of truth and what it takes for a conditional to be adequate.

\section{Conditionals, Truth and the Deduction Theorem}\label{CDT}
In natural and formal languages there is a plethora of different conditionals and before we can start answering the above question we need to clarify the scope of our investigation. To this effect we distinguish between an adequate conditional for a given logic, which we take to be key for carrying on conditional reasoning and conditionals such as indicative or subjunctive conditionals that are important from the perspective of natural language and/or philosophical semantics. Initially, our focus will be on the former, i.e., on introducing a suitable logical conditional to Kripke's theory of truth. However, as we show in Section \ref{msks} that once we have introduced such a logical conditional, it is rather straightforward to extend the framework by indicative and subjunctive conditionals.

Arguably, conditional reasoning is aptly characterized by two rules, that is, by the conditional introduction and the conditional elimination rule:
\begin{description}
\item[Conditional Introduction]If, given a set of formulas $\Gamma$, the formula $\psi$ follows from the formula $\varphi$ in the logic $L$, then the conditional $\varphi\rightarrow\psi$ follows from $\Gamma$ in $L$.
\item[Conditional Elimination]If  $\varphi\rightarrow\psi$ follows from a set of formulas $\Gamma$ in $L$, then $\psi$ follows from $\Gamma$ and $\varphi$.
\end{description}
If a logic $L$ is equipped with a conditional for which both conditional introduction and conditional elimination hold, then the logic $L$ has a conditional that respects the deduction theorem. Let $L$ be a logic and $\vDash_L$ the consequence relation of $L$. Then $\rightarrow$ satisfies the deduction theorem for $L$ iff
\begin{align*}\Gamma,\varphi\vDash_L\psi&\,\,{\rm iff}\,\,\Gamma\vDash_{L}\varphi\rightarrow\psi.\footnote{$\varphi$ and $\psi$ are formulas of the language and $\Gamma$ is a set of formulas.}  \end{align*}
On the view we endorse in this paper a conditional that violates either conditional introduction or conditional elimination in a logic $L$ does not adequately capture conditional reasoning in $L$; it is not an adequate conditional for $L$. Accordingly, a conditional is adequate conditional for a logic $L$, if and only if, it respects the deduction theorem for $L$.

Yet, a conditional that respects the deduction theorem for the given consequence relation of a logic $L$ imposes severe limitation on the properties a truth predicate can have in theories, if the consequences relation of that logic is a structural consequence relation. Let $\vDash_L$ be the consequence relation of some logic. Call $\vDash_L$ a {\it naive consequence relation} iff 
\begin{flalign*}
&(\T\text{\sc l})&&\Gamma,\varphi\vDash_L\Delta\text{ iff }\Gamma,\T\gn\varphi\vDash_L\Delta&&\\
&(\T\text{\sc r})&&\Gamma\vDash_L\varphi,\Delta\text{ iff }\Gamma\vDash_L\T\gn\varphi,\Delta&&
\end{flalign*}
for all sentences $\varphi$ and sets of sentences $\Gamma,\Delta$. Moreover, $\vDash_L$ is a structural consequence relation iff it has the following properties for $\Gamma,\Delta\subseteq{\sf Frml}$ and $\varphi\in{\sf Frml}$:\footnote{We assume a multiconclusion formulation of the consequence relation. Yet nothing hinges on this assumption. Everything we say goes through if we work with single conclusions.} 
\begin{flalign*}
&({\sf Ref})&&\Gamma, \varphi\vDash_L\varphi,\Delta;&&\\
&({\sf Weak})&&\text{if }\Gamma\vDash_L\Delta,\text{ then }\Gamma,\Gamma'\vDash_L\Delta,\Delta';&&\\
&({\sf Contr})&&\text{if }\Gamma,\varphi,\varphi\vDash_L\Delta,\text{ then }\Gamma,\varphi\vDash_L\Gamma;&&\\
&({\sf Trans})&&\text{if }\Gamma\vDash_L\varphi,\Delta\,\&\,\Gamma,\varphi\vDash_L\Delta\text{ then }\Gamma\vDash_L\Delta.&&
\end{flalign*}

The proposition below shows that we cannot have both the Deduction theorem, {\it and} a structural and naive consequence relation. This is one guise of so-called Curry's paradox:

\begin{prop}[Curry]\label{cur}Let $\vDash_L$ be a structural consequence relation for a (non-trivial) logic $L$ and $\rightarrow$ a conditional for which the deduction theorem $({\sf Ded})$ holds relative to $\vDash_L$. Then $\vDash_L$ cannot be a naive consequence relation.
\end{prop}
\begin{proof}Let $\kappa$ be the sentence $\T\gn\kappa\rightarrow\bot$. Then
\begin{enumerate}
\item $\kappa\vDash\kappa\hfill({\sf Ref})$
\item $\kappa\vDash\T\gn\kappa\rightarrow\bot\hfill 1,\text{ Def. }\kappa$
\item $\T\gn\kappa\vDash\T\gn\kappa\rightarrow\bot\hfill 2,(\T\text{\sc l})$
\item $\T\gn\kappa,\T\gn\kappa\vDash\bot\hfill 3,{\sf Ded}$
\item $\T\gn\kappa\vDash\bot\hfill 4, ({\sf Contr})$
\item $\vDash\T\gn{\kappa}\rightarrow\bot\hfill 5,{\sf Ded}$
\item $\vDash\kappa\hfill 6,\text{ Def. }\kappa$
\item $\vDash\T\gn\kappa\hfill 7,(\T\text{\sc r})$
\item $\vDash\bot\hfill 5, 8, ({\sf Trans})$
\end{enumerate}
\end{proof}
Proposition \ref{cur} is deeply troubling. It shows that we cannot have it both ways: either $\varphi$ and $\T\gn\varphi$ must somewhat come apart or we must accept a conditional that, at least from a logical perspective, is inadequate and arguably not up for the task of capturing conditional reasoning. Some theorist have taken Proposition \ref{cur} to show that we should adopt so-called substructural logics, that is, logics with consequence relations that fail to meet one of $({\sf Ref})$, $({\sf Contr})$, or $({\sf Trans})$.\footnote{The structural rule of $({\sf Weak})$ was not used in the proof of Proposition \ref{cur}.} As interesting as these approaches may be, we shall not pursue it here, but take the structural rules to be non-negotiable.\footnote{Readers interested in substructural logics are referred to \cite{pao13,res02}, and to \cite{bgr18} for an introduction to substructural ``solutions'' to the semantic paradoxes.}

A number of truth-disquotationalist have argued that truth is a logical notion that is fully characterized by the truth introduction and elimination rules.\footnote{See, e.g., \cite{bea21} for a pronounced formulations of this view.} Taking this idea seriously we can, somewhat tongue in cheek, give the following pointed formulation of the dilemma posed to us by Curry's paradox: {\it either truth is not a logical notion or there is no logically adequate conditional}. Perhaps surprisingly, in light of this dilemma many philosophers have come down in favor of the logicality of truth and have sacrificed salient properties of the conditional, that is, in most cases the rule of conditional introduction.\footnote{These include \cite{fie08,fie16,fie21}, \cite{yab03}, \cite{ros16,iaro24}, \cite{bea09}, and \cite{bac13} to name only few.} In this paper, we investigate the alternative option and develop a conditional for which the deduction theorem holds relative to a suitable non-naive consequence relation. 

In the discussion of Curry's paradox we left open how we are to understand the consequence relation of the given logic $L$. We now focus on the standard semantic definition of the consequence relation as truth preservation relative to a suitable model class, i.e., $\Gamma\vDash_L\Delta$ iff for all $M\in\mathfrak{M}_L:$
\begin{align*}\text{ if }M\Vdash\bigwedge\Gamma,&\text{ then }M\Vdash\bigvee\Delta.\end{align*}
In light of Curry's paradox we can conclude that we can only find an adequate conditional for $\vDash_L$ iff the model class $\mathfrak{M}_L$ also contains non-naive truth models, that is, models for which there are sentences $\psi$ such that $\psi$ and $\T\gn\psi$ are not semantically equivalent. However, Curry's paradox does not rule out the possibility that $\mathfrak{M}_L$ contains some naive truth models. Indeed, we shall show how to construct a naive truth model $M$ with an associated model class $\mathfrak{M}$ for a logic with an adequate conditional such that $M\in \mathfrak{M}$ and in all non-naive models $M'\in\mathfrak{M}$ the truth predicate will still satisfy desirable truth-like properties. 

The principal upshot of Curry's paradox then is that the conditional cannot express truth preservation relative to a collection of naive truth models, which, in turn, rules out the possibility of a transparent truth predicate. As a direct consequence we also know that the truth conditions of a sentence $\varphi\rightarrow\psi$ in a model $M$ cannot be given in terms of truth preservation in $M$. More generally, due to Curry's paradox the truth conditions of $\varphi\rightarrow\psi$ cannot be specified solely in terms of the truth values of $\varphi$ and $\psi$ in $M$, for otherwise the conditional would constitute a transparent context in a naive truth model, which, again, is ruled out by Curry's paradox. The upshot then is that the truth conditions of a conditional $\varphi\rightarrow\psi$ in a model $M$ need to resort to a collection of models associated with the given model $M$ and this collection needs to contain non-naive models. But how should the collection of models relevant for the evaluation of a conditional  in $M$, i.e., the models associated with $M$ be determined?

\subsection{Towards Truth-conditions}
Up to this point we have not focused on a particular logic but from now we focus on strong Kleene logic ${\sf K3}$. Our ultimate aim is to introduce an adequate conditional to ${\sf K3}$ and then show how to construct a naive truth model in the resulting logic/framework. ${\sf K3}$ allows for truth value gaps but rules out truth value gluts. This shows in the fact that explosion is a valid rule of inference for ${\sf K3}$.\footnote{A sequent calculus for ${\sf K3}$ is given in the appendix to this paper.} In light of our discussion at the end of the previous section, the key question is to determine the class of models associated with a given partial model $M$. An immediate answer is to single out those models that are semantically compatible with $M$. A model $M'$ is semantically compatible with $M$, if $M'$ is not in conflict with the semantic information provided by $M$, that is, if a formula $\varphi$ receives a classical truth value in $M$, then it will receive the same truth value in $M'$. Put differently a model is compatible with $M$ iff it provides at least as much semantic information as $M$ and disagrees with $M$ at most on those sentences that do not receive a classical truth value in $M$ . Let us write $M\leq M'$ to indicate that $M'$ is semantically compatible with $M$---notice that semantic compatibility is not a symmetric notion. Then at a first brush we propose the following truth conditions for a conditional $\varphi\rightarrow\psi$ in a model $M$:
\begin{align*}M\vDash\varphi\rightarrow\psi&\,\,{\rm iff}\text{ for all }M'\text{ s.~t. }M\geq M':\text{ if }M'\Vdash\varphi,\text{ then }M'\Vdash\psi.\end{align*}
Our talk of semantic compatibility may remind readers of supervaluational semantics where, indeed, one quantifies over models that are semantically compatibility with the given partial model. However, in orthodox supervaluational semantics one quantifies over classical models, whereas we think of semantic compatibility as a relation between partial models. In what is to come we shall generalize the above truth conditons, that is, strong Kleene supervaluation in the following two ways:
\begin{enumerate}
\item we do not quantify over all semantically compatible models but only those of a given set $X$ of models of the language;
\item we generalize the idea of semantic compatibility to that of an admissibility relation on the set $X$.
\end{enumerate}
We are ultimately led to so-called strong Kleene supervaluation structures (cf.~Section \ref{sks}) and the truth conditions for the conditional in such structures look very much like the truth conditions of an intuitionistic conditional in a Kripke frame for intuitionistic logic. Indeed, our proposal can be seen as introducing an intuitionistic conditional to strong Kleene logic ${\sf K3}$. It turns out that the resulting logic is known as three-valued Nelson logic {\sf N3} \citep[see, e.g.~][]{tho69,wan91,wan01} and enjoys a soundness and completeness result relative to a constant domain Kripke semantics for intuitionistic logic \citep[see][]{tho69}. Importantly, the deduction theorem holds for the ${\sf N3}$-conditional. Of course, it then remains to show how to find naive truth models. Before we turn to introducing strong Kleene supervaluation in earnest, we repeat a point made in the Introduction to this paper, namely, that strong Kleene supervaluation serves as a general framework for obtaining adequate truth conditions for non-monotone notions in the presence of semantic indeterminacy. For example, as we discuss elsewhere, it can be applied to obtain a neat semantics for $\langle 1,1\rangle$-quantifiers in partial semantics. Yet, as mentioned, in this paper we focus our attention on conditionals.

\section{Strong Kleene Supervaluation}\label{sks}
In this section we introduce the formal framework. We start by introducing the formal syntax of our language $\mathcal{L}$. We assume the language to consist of
\begin{itemize}
\item a denumerable set of variables $v_0,v_1,v_2\ldots$ denoted ${\sf Var}_\mathcal{L}$;
\item a denumerable set of individual constants $c_0,c_1,c_2,\ldots$ denoted ${\sf Cons}_\mathcal{L}$;
\item a denumerable set of function symbols $f^{n_0}_0,f^{n_1}_1,f^{n_2}_2,\ldots$ with denoted ${\sf Funct}_\mathcal{L}$ such that $n_j$ denotes the arity of $f_j$ with $n_j,n,j\in\omega$;
\item a denumerable set of function symbols $P^{n_0}_0,P^{n_1}_1,P^{n_2}_2,\ldots$ with denoted ${\sf Pred}_\mathcal{L}$ such that $n_j$ denotes the arity of $P_j$ with $n_j,n,j\in\omega$;
\item the identity symbol $=$;
\item the logical vocabulary $\neg,\wedge,\rightarrow$.
\end{itemize}

\begin{de}[Term, Formula]
The set of well-formed terms ${\sf Term}_\mathcal{L}$ is specified as follows
$$t_1,\ldots,t_n::=v_j\,|\,c_k\,|\,f_m^n(t_1,\ldots,t_n)$$
with $n,j,k,m\in\omega$, $v_j\in{\sf Var}_\mathcal{L}$, $c_k\in{\sf Cons}_\mathcal{L}$, and $f^n_m\in {\sf Funct}_\mathcal{L}$. The set of well-formed formulas ${\sf Form}_\mathcal{L}$ is specified as follows:
$$\varphi::=(s=t)\,|\,P_j^m(t_1,\ldots,t_m)\,|\,\neg\varphi\,|\,\varphi\wedge\varphi\,|\,\varphi\rightarrow\varphi\,|\,\forall x\varphi$$
with $j,m\in\omega$, $s,t,t_1,\ldots,t_m\in{\sf Term}_\mathcal{L}$, and $x\in{\sf Var}_\mathcal{L}$.
\end{de}
We shall make use of further logical symbols which, however, we conceive of mere notational abbreviations. In particular, we set
\begin{align*}
&\varphi\vee\psi:=\neg(\neg\varphi\wedge\neg\psi)\\
&\exists x\varphi:=\neg\forall x\neg\varphi.
\end{align*}

With the syntax in place we start introducing the semantics and the notion of a (strong Kleene) supervaluation structure which will consist of a non-empty domain and set $X$ of interpretation functions ordered by an admissibility relation $H$. First, we specify the notion of a (partial) interpretation on a non-empty domain.

\begin{de}[Interpretation, Ordering]\label{kinto}Let $D\neq\emptyset$. A function $I$ for some language $\mathcal{L}$ is called a ${\sf k}$-interpretation iff
\begin{enumerate}[label=(\roman*)]
\item $I(t)\in D$ for all $t\in {\sf Const}_\mathcal{L}$;
\item $I(f^n)\in^{D^n}\!\!D$ for all $f\in{\sf Funct}_\mathcal{L}$;
\item $I(P^n)\in \mathcal{P}(D^n)\times\mathcal{P}(D^n)$;
\item $I^+(P^n)\cap I^-(P)=\emptyset$.
\end{enumerate}
The set of all interpretations over $D$ is denoted by ${\sf Int}_D$. For $I,J\in{\sf Int}_D$ we say $I\leq J$ iff $I^+(P)\subseteq J^+(P)$ and $I^-(P)\subseteq J^-(P)$ for all $P\in{\sf Pred}$.
\end{de}

\begin{de}[Supervaluation structure]\label{SVS}A tuple $\mathfrak{M}=(D,X,H)$ is called a supervaluation structure iff
\begin{enumerate}[label=(\roman*)]
\item $D$ is a nonempty domain of individuals;
\item $X\subseteq{\sf Int}_D$ such that for all $J,I\in X$
\begin{itemize}
\item $J(c)=I(c)$ for $c\in{\sf Const}_{\mathcal{L}}$;
\item $J(f)(d_1,\ldots,d_n)=I(f)(d_1,\ldots,d_n)$ for $f\in{\sf Func}_\mathcal{L}$, $d_1,\ldots,d_n\in D$, and $n\in\omega$;
\end{itemize}
\item $H\subseteq X\times X$ such that
\begin{itemize}
\item $H$ is reflexive and transitive;
\item if $(x,y)\in H$, then $x\leq y$.
\end{itemize}
\end{enumerate}
\end{de}
It remains explaining the notion of truth in a supervalutation structure relative to a given interpretation. To this effect we need to appeal to variable assignments.

\begin{de}[Assignment]\label{ass}Let $D\neq\emptyset$. An assignment on $D$ is a function $\beta:{\sf Var}\rightarrow D$ that assigns to the variables of the language objects in $D$. Let $\beta$ be an assignment and $d\in D$. Then $\beta(x:d)$ such that 
\begin{align*}\beta(x:d)(v):=\begin{cases}d,&\text{ if }x\doteq v\\\beta(v),&\text{otherwise}.\end{cases}\end{align*}
is an $x$-variant of $\beta$.\end{de}

We now specify the denotation of a term in a supervaluation structure relative to an assignment.

\begin{de}[Denotation in a supervaluation structure]\label{Den}Let $\mathfrak{M}=(D, X,H)$ be a supervaluation structure and $\beta$ an assignment on $F$. Then for $J\in X$ the denotation of a term $t$ under the assignment $\beta$ is defined as follows:
\begin{align*}J^\beta(t):=\begin{cases}\beta(t),&\text{ if }t\in{\sf Var}_\mathcal{L}\\J(t),&\text{ if }t\in{\sf Const}_\mathcal{L}\\J(f^n)(J^\beta(t_1),\ldots,J^\beta(t_n)),&\text{ if }t\doteq f^n(t_1,\ldots,t_n).\end{cases}\end{align*}
\end{de}
Notice that $J^\beta(t)=I^\beta(t)$ for all $J,I\in X$ and variable assignments $\beta$.

\begin{de}[Truth in supervaluation structure]\label{TSV}Let $\mathfrak{M}=(D,X,H)$ be a supervaluation structure with $J\in X$ and $\beta$ a variable assignment. Then a formula $\varphi$ is true in $\mathfrak{M}$ at $J$ under assignment $\beta$ ($\mathfrak{M},J\Vdash\varphi[\beta]$) iff
\begin{align*}
&J^{\beta}(t_1),\ldots,J^{\beta)}(t_n)\rangle\in J^+(P^n),&&\text{ if }\varphi\doteq (P^n t_1,\ldots,t_n);\\
&J^{\beta}(t_1),\ldots,J^{\beta}(t_n)\rangle\in J^-(P^n),&&\text{ if }\varphi\doteq (\neg P^n t_1,\ldots,t_n);\\
&J^{\beta} (s)= J^{\beta} (t),&&\text{ if }\varphi\doteq(s=t);\\
&J^{\beta} (s)\neq J^{\beta} (t),&&\text{ if }\varphi\doteq(\neg(s=t));\\
&\mathfrak{M},J\Vdash\psi[\beta]&&\text{ if }\varphi\doteq\neg\neg\psi;\\
&\mathfrak{M},J\Vdash\psi[\beta]\,\&\,\mathfrak{M},J\Vdash\chi[\beta],&&\text{ if }\varphi\doteq(\psi\wedge\chi);\\
&\mathfrak{M},J\Vdash\neg\psi[\beta]\text{ or }\mathfrak{M},J\Vdash\neg\chi[\beta],&&\text{ if }\varphi\doteq(\neg(\psi\wedge\chi));\\
&\forall J'((J,J')\in H\Rightarrow \mathfrak{M},J\not\Vdash\psi[\beta]\text{ or }\mathfrak{M},J\Vdash\chi[\beta],&& \text{ if }\varphi\doteq(\psi\rightarrow\chi));\\
&\mathfrak{M},J\Vdash\psi[\beta]\,\&\,\mathfrak{M},J\Vdash\neg\chi[\beta],&& \text{ if }\varphi\doteq\neg(\psi\rightarrow\chi);\\
&\mathfrak{M},J\Vdash\psi[\beta(x:d)]\text{ for all }d\in D,&&\text{ if }\varphi\doteq\forall x\psi;\\
&\mathfrak{M},J\Vdash\neg\psi[\beta(x:d)]\text{ for some }d\in D,&&\text{ if }\varphi\doteq\neg\forall x\psi.
\end{align*}
We say that $\varphi$ is true at $J$ $(\mathfrak{M},J\Vdash\varphi)$ iff $\mathfrak{M},J\Vdash\varphi[\beta]$ for all assignments $\beta$. Finally, a formula is said to be true in $\mathfrak{M}$ ($\mathfrak{M}\Vdash\varphi$) iff $\mathfrak{M},J\Vdash\varphi$ for all $J\in X$.
\end{de}

As mentioned  previously truth in a supervaluation structure as specified in Definition \ref{TSV} coincides with truth in Kripke frames for constant domain Nelson logic ${\sf N3}$ \citep{tho69,wan91} with classical identity.\footnote{For a presentation of the various logics appealed to in this paper in form of sequent calculi we refer the reader to Appendix \ref{SeqC}. As a guideline we point out that if $\Gamma\Rightarrow\Delta$ is a valid ${\sf K3}$-sequent, then $\bigwedge\Gamma\rightarrow\bigvee\Delta$ is an ${\sf N3}$-logical truth, where $\Gamma$ and $\Delta$ are finite sets of formulae.}

Thus, unsurprisingly, we obtain the following monotonicity/persistence property.

\begin{lem}\label{pers}Let $\mathfrak{M}=(D,X,H)$ be a supervaluation structure, $\beta$ a variable assignment, and $J,J'\in X$ such that $(J,J')\in H$.Then for all $\varphi\in\mathcal{L}$
\begin{align*}\text{if }\mathfrak{M},J\Vdash\varphi[\beta],\text{ then }\mathfrak{M},J'\Vdash\varphi[\beta].\end{align*}
\end{lem}

\begin{proof}The claim is established by an induction on the positive complexity of $\varphi$. The only interesting case is that of the conditional. The claim follows because due to the transitivity of $H$ we obtain:
$$\{I\in X\,|\,(J',I)\in H\}\subseteq\{I\in X\,|\,(J,I)\in H\}.$$
\end{proof}

The main difference between the Nelson logic and the present framework is that in our case interpretations serve, in some sense, a double purpose: they serve as index/parameter in the semantic explanation whilst, at the same time, providing the interpretation of the descriptive vocabulary. In contrast, in Kripke frames these two roles are kept apart. Worlds serve as parameter in the semantic explanation, whilst an interpretation on the frame determines the extension of the descriptive vocabulary relative to a world. Indeed, this seems to be a good general characterization of the conceptual difference between many forms of supervaluation and  possible world semantics.

 Nelson logic can be understood as extending strong Kleene logic ${\sf K3}$ by an intuitionistic conditional and fittingly the Nelson conditional does not allow for contraposition.\footnote{In ${\sf K3}$ the sequent $\varphi,\neg\varphi\Rightarrow\bot$ is provable, but it's contraposition $\top\Rightarrow\varphi,\neg\varphi$ is not ${\sf K3}$-valid.} The failure of contraposition is not a mere accident but a consequence of the extensional (local) refutation condition of the conditional and double negation elimination: if contraposition is introduced to the framework then a formula $\varphi\rightarrow\psi$ would be equivalent to $\neg\varphi\vee\psi$ which would collapse ${\sf N3}$ into classical logic:
\begin{enumerate}
\item $\neg(\varphi\rightarrow\varphi)\leftrightarrow\varphi\wedge\neg\varphi$\hfill by truth conditions
\item $\neg\neg(\varphi\rightarrow\varphi)\leftrightarrow\neg(\varphi\wedge\neg\varphi)$\hfill contraposition
\item $\varphi\rightarrow\varphi\leftrightarrow\varphi\vee\neg\varphi$\hfill double negation intro/elim; De Morgan
\item $\varphi\vee\neg\varphi$
\end{enumerate}
 The last inference follows from the fact that $\varphi\rightarrow\varphi$ is an ${\sf N3}$-logical truth and modus ponens. However, from the perspective of ${\sf K3}$ the failure of contraposition seems the right outcome, as the consequence relation of ${\sf K3}$ does not admit contraposition. However, since the deduction theorem holds for the ${\sf N3}$-conditional the above argument also shows that ${\sf N3}$ and thus strong Kleene supervaluation are hyperintensional and, indeed, not self-extensional in the sense of \cite{odwa21}: despite the fact that $\neg(\varphi\rightarrow\psi)$ and $\varphi\wedge\neg\psi$ are interderivable in ${\sf N3}$ they cannot be substituted {\it salva veritate} in every context. From our perspective the moral of this observation is that strong Kleene supervaluation forces us to take bilateralism seriously: the proposition expressed by a sentence is a tuple consisting of the truth- and the falsemakers of the sentence, where the set of falsemakers cannot be retrieved from the set of truthmakers of a sentence. From this bilateral perspective ${\sf N3}$ ceases to be hyperintensional in the following sense: if two sentences express the same semantic content they can be substituted {\it salva veritate} in every context in ${\sf N3}$. Adopting the aforementioned perspective we propose to revise Odintsov and Wansing's criterion for self-extensionality: a bilateral logic is {\it self-extensional} iff for formulas $\varphi,\psi,$ and $\theta(p)$:
 \begin{align*}\text{if }\varphi\dashv\vdash\psi\text{ and }\neg\varphi\dashv\vdash\neg\psi,\text{ then }\theta(\varphi/p)\dashv\vdash(\psi/p.)\end{align*}
 According to this revised criterion ${\sf N3}$ (and thus strong Kleene supervaluation) are self-extensional and in this sense not hyperintensional.\footnote{Unsurprisingly, ${\sf N3}$ will still be hyperintensional in the sense of \cite{cre75}.}
 
To be sure, the asymmetry between the global/intensional truth conditions of the conditional and it's local/extensional falsity condition seem to be a downside to the otherwise attractive ${\sf N3}$-conditional. Is it possible to give alternative falsity conditions for conditionals, that is, falsity conditions that are global (extensional) like the truth conditions of the conditional? One idea that comes to mind, inspired by the falsification condition of the intuitionistic conditional, takes a conditional to be false at a given interpretation iff for every interpretation it can ``see'' an admissible interpretation which provides a counterexample to the truth of the conditional: 
\begin{align*}\mathfrak{M},J\Vdash\neg(\varphi\rightarrow\psi)[\beta]&\text{ iff }\forall J'(J,J)\in H\Rightarrow\exists J''((J',J'')\in H\,\&\,\mathfrak{M},J''\Vdash\varphi\,\&\,\mathfrak{M},J''\Vdash\neg\psi)).\footnotemark\end{align*}
\footnotetext{Given these truth conditions a negated conditional would be true of $\varphi\wedge\neg\psi)$ is an element of the intrinsic fixed point. Alternatively, one could also work with truth conditions in which one does not require $\psi$ to be false but merely that it will never be true:
\begin{align*}\mathfrak{M},J\Vdash\neg(\varphi\rightarrow\psi)[\beta]&\text{ iff }\forall J'(J,J)\in H\Rightarrow\exists J''((J',J'')\in H\,\&\,\mathfrak{M},J''\Vdash\varphi\,\&\,\mathfrak{M},J''\not\Vdash\psi)).\end{align*}}
If such falsity conditions were adopted one might even try to work over symmetric strong Kleene logic ${\sf KS3}$ and obtain a conditional that respects contraposition.\footnote{${\sf KS3}$ is our label for the logic thaat result from intersecting ${\sf K3}$ and ${\sf LP}$. See \cite{nist21} for a ue of this label and a presentation of the semantics and proof-theory of the logic.} Presumably the resulting logic would amount to a partial counterpart to Leitgeb's \citeyearpar{lei19} paraconsistent logic HYPE, i.e., the logic ${\sf QN}^{*}$ \cite[cf.][]{odwa21}. However, as Nelson logic ${\sf N3}$ is well understood and has many attractive properties we leave further discussion of alternate falsification conditions for the conditional and the resulting logics for another occasion.\footnote{We also note that at least upon first glance it seems that the lower bound of the recursion theoretic complexity of the interpretation of the naive truth predicate would increase from $\Pi^1_1$ to $\Pi^1_2$ on the revised falsity condition of the conditional.}

This completes the introduction of strong Kleene supervaluation structures for an arbitrary first-order language. The next step is to show how to extend the framework to a langauge with a self-applicable truth predicate and, ultimately, to construct naive truth models. From Curry's paradox we know that within a strong Kleene supervaluation structure there must be interpretations at which the truth predicate is not naive, yet we can find attractive interpretations at which the truth predicate is naive. In sum, truth will be locally but not globally naive. Moreover, in addition we can make sure that the truth predicate has certain desirable properties globally, and on our view we thus arrive at an attractive truth theory.


\section{Truth Structures}\label{TS}
Having introduced the general framework of strong Kleene supervaluation we now start constructing supervaluation structures, which, relative to a set of designated interpretations, yield attractive interpretations of a language with an objectlinguistic, self-applicable truth predicate. We continue working with a first-order language $\mathcal{L}$, but in contrast to Section \ref{sks} we need to make some assumptions regarding the interpretation of the language. As we shall require names for the expressions of the language and also want to be able to talk about substitutions of expressions, we need our language to comprise syntactic vocabulary. We also need to make sure that this vocabulary is interpreted in the intended way. To this effect we assume that the language $\mathcal{L}$ and an associated supervaluation structure $\mathfrak{M}=(D, X,H)$ satisfy the following requirements:
\begin{itemize}
\item $\mathcal{L}$ extends the language of some syntax theory $\mathcal{L}_S$, e.g., the language of arithmetic;
\item for each $d\in D$ there exists a $c\in{\sf Const}_\mathcal{L}$ such that $J(c)=d$ and for all $\varphi\in\mathcal{L}_S$; $J\in X$ and assignment $\beta$:
$$\text{for all }c\in{\sf Const}_\mathcal{L}(\mathfrak{M},J\Vdash\varphi(c)[\beta])\text{ iff }\mathfrak{M},J\Vdash\forall x\varphi(x/c)[\beta];$$
\item for all $\varphi\in\mathcal{L}_S$; $J,J'\in X$ and assignments $\beta$.
\begin{enumerate}[label=(\roman*)]
\item $\mathfrak{M},J\Vdash\varphi[\beta]\text{ iff }\mathfrak{M},J'\Vdash\varphi[\beta]$ 
\item $\mathfrak{M},J\Vdash\varphi\vee\neg\varphi[\beta]$
\end{enumerate}
\end{itemize}

The language $\lt$ extends $\mathcal{L}$ by introducing an additional one-place predicate constant, that is, the truth predicate $\T$. 

Given a supervaluation structure for  $\mathcal{L}$ we need to provide an interpretation for the truth predicate relative to every $J\in X$. To this effect  we introduce valuation functions, that is, functions that assign an interpretation to the truth predicate relative to an interpretation.

\begin{de}[Basic Valuations]\label{adval} A valuation on a supervaluation structure $\mathfrak{M}=(D,X,H)$ is a function $f:X\rightarrow\mathcal{P}({\sf Sent}_{\lt})$. The set of all valuation on $\mathfrak{M}$ is denoted by ${\sf Val}_\mathfrak{M}$. A valuation function is basic for $\mathfrak{M}$ iff
\begin{enumerate}[label=(\roman*)]
\item $f$ is consistent, that is, $f(J)\cap\overline{f(J)}=\emptyset$ for all $J\in X$ where
$$\overline{f(J)}:=\{\varphi\,|\,\neg\varphi\in f(J)\};$$
\item for all $J\in X$ and $\varphi\in\mathcal{L}$:
\begin{align*}\text{if }\varphi\in f(J),&\text{ then }\mathfrak{M},J\Vdash\varphi\end{align*}
\item for all $J,J'\in X$, if $(J,J')\in H$, then $f(J)\subseteq f(J')$.
\end{enumerate}
The set of basic valuation functions for a given model $\mathfrak{M}$ is denoted by $\mathcal{B}_\mathfrak{M}$.\end{de}
In short, a valuation is deemed basic, if (i) it is consistent, (ii) does not declare $\mathcal{L}$-sentences true at an interpretation, unless they are made true in the underlying supervaluation structure at that interpretation, and if (iii) the valuation respects the ordering of the interpretations in $X$.  The next task is to define an ordering on the basic valuations and to introduce the idea of the {\it admissibility condition} for basic valuations

\begin{de}[Ordering, Admissibility condition]\label{ordac}
Let $\leq\subseteq{\sf Val}_\mathfrak{M}\times{\sf Val}_\mathfrak{M}$ be an ordering on ${\sf Val}_\mathfrak{M}$ defined as follows:
\begin{align*}f\leq g&:\leftrightarrow\forall J\in X(f(J)\subseteq g(J))\end{align*}
An admissibility condition is a function $\Phi:{\sf Val}_\mathfrak{M}\rightarrow\mathcal{P}(\mathcal{B}_\mathfrak{M})$  that yields a set of basic valuation functions for every member of ${\sf Val}_{\mathfrak{M}}$ such that 
\begin{itemize}
\item $\Phi(f)=\emptyset$, if $f\not\in\mathcal{B}_\mathfrak{M}$;
\item if $f\leq g$ and $f\in\mathcal{B}_\mathfrak{M}$, then $\Phi(g)\subseteq\Phi(f)$;
\item if $g\in\Phi(f)$, then $f\leq g;$
\end{itemize}
for all $f,g\in{\sf Val}_\mathfrak{M}$. For all $f,g\in\mathcal{B}_\mathfrak{M}$ the partial order $\leq_\Phi\subset\mathcal{B}_\mathfrak{M}\times\mathcal{B}_\mathfrak{M}$ is defined as follows.
\begin{align*}f\leq_\Phi g&:\leftrightarrow g\in\Phi(f).\end{align*}
\end{de}
Notice that by definition $\leq_{\Phi}$ is transitive, i.e., for $f,g,h\in{\sf Val}_\mathfrak{M}$, if $f\leq_\Phi g$ and $g\leq_\Phi h$, then $f\leq_\Phi g$. 

Together, an interpretation for the language $\mathcal{L}$ and a valuation function yield an interpretation function for the language $\lt$. 
\begin{de}[Truth interpretation, Admissibility relation]\label{adrel}
Let $J\in X$ and $f\in\mathcal{B}_\mathfrak{M}$, then $J_f$ is a called a truth-interpretation for the language $\lt$. We have
\begin{align*}J_f(P):=\begin{cases}f(J),&\text{ if }P\doteq\T;\\J(P),&\text{ otherwise}.\end{cases}\end{align*}
We may then define an admissiblity relation on $X\times\mathcal{B}_\mathfrak{M}$ for truth interpretations on the basis of $H$ and $\leq_\Phi$, that is, for all $I,J\in X$ and $f,g\in{\sf Val}_{\mathfrak{M}}$
\begin{align*}(I_f,J_g)\in H_\Phi&:\leftrightarrow (I,J)\in H\,\&\,f\leq_\Phi g.\end{align*}

\end{de}

Having defined an admissibility relation for truth interpretations we can introduce the notion of a truth structure, that is, a supervaluation structure for the language $\lt$.

\begin{de}[Truth Structure, Grounded Truth Structure]\label{dets}Let $\mathfrak{M}=(D,X,H)$ be a supervaluation structure for $\mathcal{L}$. Then the tuple $(D,X\times Y,H_\Phi)$ with $Y\subseteq\mathcal{B}_\mathfrak{M}$ and where $H_\Phi$ restricted to $X\times Y$ is called a truth structure over $\mathfrak{M}$. If $Y\subseteq\mathcal{B}_\mathfrak{M}$ has a $\leq$-minimal valuation function, that is, an $f$ in $Y$ such that $f\leq g$ for all $g\in Y$ and $Y\cap\Phi(f)\neq\emptyset$, then $(D,X\times Y,H_\Phi)$ is called a grounded truth structure and $Y$ is called a grounded set of valuations. The set of all grounded sets of valuations is denoted by ${\sf Adm}_\mathfrak{M}$.
\end{de}

If $\leq_\Phi$ is reflexive on $Y$ for a given truth structure/grounded truth structure $\mathfrak{M}_{\T}$, then $\mathfrak{M}_{\T}$  is a supervaluation structure for the language $\lt$.\footnote{So, strictly speaking, truth structures are not supervaluation structure for $\lt$. However, as we discuss in Section \ref{KTS}, given a reflexive truth-interpretation $J_f$, one can obtain a supervaluation structure for $\lt$ by constructing the $J_f$-generated substructure of the given truth structure.} Yet, we are not guaranteed that the predicate `$\T$' has desirable, truth-like properties in $\mathfrak{M}_{\T}$: a truth structure places only minimal constraints on the interpretation of the truth predicate. In particular, we are not guaranteed that there is a truth interpretation $J_f$ such that for all $\lt$-sentence $\varphi$:
\begin{align*}\mathfrak{M}_{\T},J_f\Vdash\varphi&\text{ iff }\mathfrak{M}_{\T},J_f\Vdash\T\gn\varphi.\end{align*} 
To construct (reflexive) truth structures for which there is a designated set of truth interpretations that satisfy the modified version of convention $\T$ will be the goal of the next section. However, before we turn to this task we prove two general monotonicity properties of truth structures.

\begin{lem}[$J$-monotonicity]\label{gmon1}Let $(D,X\times Y,H_\Phi)$ be a truth structure. Then for $J,J'\in X$, $f\in Y$, assignments $\beta$ and all $\varphi\in\lt$: if $(J,J')\in H$, then
\begin{align*}{\rm if}\,((D,X\times Y,H_\Phi),J_f\Vdash\varphi[\beta]&\,{\rm then}\,(D,X\times Y,H_\Phi),J'_f\Vdash\varphi[\beta]).\end{align*}
\end{lem}

\begin{proof}
The only interesting case is that of the conditional. Notice that there are two cases: (i) there is a $g\in Y$ with $f\leq_\Phi g$ and (ii) there exists no $g\in Y$ such that $f\leq_\Phi g$. In the latter case all conditionals are true at both $J_f$ and $J_{f}'$ as no truth-interpretation will be accessible from either $J_f$ or $J_{f}'$ (by definition of $H_\Phi$) and the claim holds trivially. For  case (i), the reasoning is as for Lemma \ref{pers} since $H_\Phi$ is transitive.
\end{proof}

\begin{lem}[$(Y,f)$-monotonicity]\label{gmon2}Let $Z\subseteq Y\subseteq\mathcal{B}_\mathfrak{M}$ and $f\in Y$ and $g\in Z$. Then for all $J\in X$ assignments $\beta$ and $\varphi\in\lt$: if $f\leq g$, then
\begin{align*}{\rm if}\,(D,X\times Y,H_\Phi),J_f,\Vdash\varphi[\beta],&\,{\rm then}\,(F,X\times Z,H_\Phi),J_g\Vdash\varphi[\beta].\end{align*}\end{lem}

\begin{proof}Again the only interesting case is the conditional, but the claim follows since by Definition \ref{adrel}, $f\leq g$ implies $\Phi(g)\subseteq\Phi(f)$. Then by transitivity $(J_f,J'_k)\in H_\Phi$ for all $J'\in X$ and $k\in Y$, if $(J_g,J'_k)\in H_\Phi$.\end{proof}

\section{Kripkean Truth Structures}\label{KTS}
Let $\mathfrak{M}=(D,X,H)$ be a supervaluation structure of $\mathcal{L}$. We now wish to construct a grounded truth structure $(D,X\times Y_f,H_\Phi)$ ($f$ is the minimal element of $Y$) such that for all $J\in X$ and all $\lt$-sentences:
\begin{align*}(D,X\times Y_f,H_\Phi),J_f\Vdash\varphi&\,\,{\rm iff}\,\,(D,X\times Y_f,H_\Phi),J_f\Vdash\T\gn\varphi.\end{align*}
The rough idea is that the truth interpretations $J_f$ are models of naive truth, while the truth interpretation $J_g$ with $g>f$ constitute first-order models of $\lt$ that are compatible with the semantic information given by $J_f$ and thus relevant for evaluating the conditional connective. However, these models may not be naive truth models.

We start the construction by defining an operation $\theta^{\Phi}_\mathfrak{M}:\mathcal{P}(\mathcal{B}_\mathfrak{M})\times\mathcal{B}_\mathfrak{M}\rightarrow{\sf Val}_{\mathfrak{M}}$ which maps valuation functions onto valuation functions relative to a truth model based on $\mathfrak{M}$. For $Y\subseteq\mathcal{B}_\mathfrak{M}$ and $f\in Y$ the operation $\theta$ yields a new valuation function such that for all $J\in X$
\begin{align*}[\theta^{\Phi}_\mathfrak{M}(Y,f)](J)&:=\{\varphi\in{\sf Sent}_{\lt}\,|\,(D,X\times Y,H_\Phi),J_f\Vdash\varphi\}.\footnotemark\end{align*}
\footnotetext{Strictly speaking $\theta$ is thus only defined for pairs $(Y,f)$ with $f\in Y$.}\vspace{-3ex}

\begin{lem}\label{wdef}Let $\mathfrak{M}=(D,X,H)$ be an arbitrary supervaluation structure for $\mathcal{L}$ and $f\in Y\in\mathcal{P}(\mathcal{B}_\mathfrak{M})$. Then $\theta^\Phi_\mathfrak{M}$ is well-defined on $\mathcal{B}_\mathfrak{M}$, if $\Phi(f)\cap Y\neq\emptyset$, that is, 
$$\text{if }f\in Y\,\&\,\Phi(f)\cap Y\neq\emptyset,\text{ then }\theta^\Phi_\mathfrak{M}(Y,f)\in\mathcal{B}_\mathfrak{M}.$$ 
\end{lem}
\begin{proof}We need to show that $\theta^\Phi_\mathfrak{M}(Y,f)$ satisfies condition (i-iii) of Definition \ref{adval}. Since $f\in Y\subseteq\mathcal{B}_\mathfrak{M}$
and $\Phi(f)\cap Y\neq\emptyset$  (i) and (ii) are immediate. (iii) is a direct consequence of Lemma \ref{gmon1}
\end{proof}

From now on we restrict our attention to grounded truth structures $(D,X\times Y_f,H_\Phi)$, that is, truth structures for which there is $\Phi$-minimal $f\in Y_f$ with $\Phi(f)\cap Y_f\neq\emptyset$. By employing the operation $\theta^\Phi_\mathfrak{M}$ we define a second operation $\Theta^\Phi_\mathfrak{M}:\mathsf{Adm}_\mathfrak{M}\rightarrow\mathcal{P}(\mathcal{B}_\mathfrak{M})$ on grounded sets of valuations  such that for $Y_f\in\mathsf{Adm}_\mathfrak{M}$:
\begin{align*}\Theta^\Phi_\mathfrak{M}(Y_f)&:=\{g\in Y_f\,|\,\theta^\Phi_\mathfrak{M}(Y_f,f)\leq g\}.\end{align*}

Intuitively, the operation $\theta^\Phi_\mathfrak{M}$ will be used to construct the minimal fixed point, while the operation $\Theta^\Phi_\mathfrak{M}$ determines the truth interpretations that are compatible with the semantic information provided and thus relevant for evaluating the conditional. We note that $\Theta^\Phi_\mathfrak{M}$ is monotone and decreasing, i.e., 
\begin{flalign*}&{\rm Monotonicity:}&&\text{if }Y_f\subseteq Z_g\text{ then }\Theta^\Phi_\mathfrak{M}(Y_f)\subseteq \Theta^\Phi_\mathfrak{M}(Z_g),&&\\
&{\rm Decreasing:}&&\Theta^\Phi_\mathfrak{M}(Y_f)\subseteq Y_f&&\end{flalign*}
for $Y_f,Z_g\in{\sf Adm}_\mathfrak{M}$.

Notice however that for arbitrary $Y_f\in {\sf Adm}_\mathfrak{M}$ we are not guaranteed that $\Theta^\Phi_\mathfrak{M}(Y_f)\in{\sf Adm}_\mathfrak{M}$: for one, it is possible that $\theta^\Phi_\mathfrak{M}(Y_f,f)\not\in Y_f$ and thus, possibly, $\Theta^\Phi_\mathfrak{M}(Y_f)$ may lack a minimal element. For another it is possible that $\Theta^\Phi_\mathfrak{M}(Y_f)=\emptyset$. In short, in both cases $\Theta^\Phi_\mathfrak{M}(Y_f)$ may not be a grounded truth set.

Ultimately we want to show the existence of non-trivial fixed points of $\theta^\Phi_{\mathfrak{M}}$ and $\Theta^\Phi_\mathfrak{M}$ for suitable admissibility conditions $\Phi$, i.e., $f\in Y_f\in{\sf Adm}_\mathfrak{M}$ (and thus $Y_f\neq\emptyset$) such  that
\begin{align*}&\theta^\Phi_{\mathfrak{M}}(Y_f,f)=f\\
&\Theta^\Phi_\mathfrak{M}(Y_f)=Y_f.\end{align*}
In turns out that in the event that $\theta^\Phi_{\mathfrak{M}}(Y_f,f)\in Y_f$, the existence of a fixed point $Y_f$ of $\Theta^\Phi_\mathfrak{M}$ implies that $\theta^\Phi_\mathfrak{M}(Y_f,f)=f$, i.e., that $f$ is a fixed point of $\theta$ relative to the grounded truth set $Y_f$. More precisely, for all  $J\in X$ $f(J)$ would be the minimal fixed point over $J$ relative to set of admissible valuation $Y_f$ 

\begin{lem}\label{AFL}Let $Y_f\in{\sf Adm}_\mathfrak{M}$ with $\theta^\Phi_{\mathfrak{M}}(Y_f,f)\in Y_f$. Then 
\begin{align*}\Theta^\Phi_{\mathfrak{M}}(Y_f)=Y_f\text{ iff }\theta^\Phi_\mathfrak{M}(Y_f,f)=f.\end{align*} 
\end{lem}


Our strategy for finding fixed points will be to start from a suitable grounded truth set with minimal valuation function $f$and then, by consecutively and simultaneously applying $\theta$ and $\Theta$, construct the minimal fixed point of $\theta$ whilst simultaneously sieving out unsuitable valuation functions until we end up with a fixed point of $\Theta$. To spell out this idea in a rigorous fashion we define iterative applications of $\Theta^\Phi_\mathfrak{M}$ to some set $Y_f\in{\sf Adm}$ with minimal valuation function $f$ by transfinite recursion on the ordinals and, simultaneously, iterative applications of  $\theta^\Phi_\mathfrak{M}$ to the pair ($Y_f,f)$:

\begin{align*}\theta^\alpha(Y_f,f)&:=\begin{cases}f,&\text{ if }\alpha=0;\\\theta(\Theta^\beta(Y_f),\theta^\beta(Y_f,f)),&\text{ if }\alpha=\beta+1;\\\bigcup_{\beta<\alpha}\theta^\beta(Y_f,f),&\text{ if $\alpha$ is limit.}\end{cases}\end{align*}

\begin{align*}\Theta^\alpha(Y_f)&:=\begin{cases}Y_f,&\text{ if }\alpha=0;\\\{g\in \Theta^\beta(Y_f)\}\,|\,\theta^\alpha(Y_f,f)\leq g\},&\text{ if }\alpha=\beta+1\text{ and }\Theta^\beta(Y_f)\in{\sf Adm}_\mathfrak{M};\\
\emptyset,&\text{ if }\alpha=\beta+1\text{ and }\Theta^\beta(Y_f)\not\in{\sf Adm}_\mathfrak{M};\\
\{g\in\bigcap_{\beta <\alpha}\Theta^\beta(Y_f)\,|\,\theta^\alpha(Y_f,f)\leq g\},&\text{ if $\alpha$ is limit.}\end{cases}\end{align*}

Since $\Theta^\Phi_\mathfrak{M}$ is monotone and decreasing, we can employ a standard cardinality argument to show that for all $Y_f\in{\sf Adm}_\mathfrak{M}$ there is an $\xi$ such that $\Theta^\xi(Y_f)=\Theta^\alpha(Y_f)$ for all $\alpha\geq\xi$. Unfortunately, this falls short from establishing the existence of non-trivial fixed points of $\Theta^\Phi_\mathfrak{M}$ for it may turn out that $\Theta^\xi(Y_f)=\emptyset$, which implies $\theta^{\xi+1}(Y_f,f)={\sf Sent}$. This can happen because (i) for ordinal $\alpha\leq\xi$ the set $\Theta^\alpha(Y_f)\neq\emptyset$ lacks a minimal element, i.e., $\Theta^\alpha(Y_f)\not\in{\sf Adm}_\mathfrak{M}$ or, alternatively, if for some limit ordinal $\kappa\leq\xi$ there exists no $g\in \bigcap_{\beta <\kappa}\Theta^\beta(Y_f)$ such that $\theta^\kappa(Y_f,f)\leq g$. So in both cases the $\theta$-operation has moved us outside the range of available interpretations of the truth predicate albeit for different reasons. Indeed this may happen even if  $\Theta^\xi(Y_f)\neq\emptyset$, that is, in case we have found a fixed point of $\Theta$ but $\theta^\xi(Y_f,f)\not\in\Theta^\xi(Y_f)$. This shows that finding a fixed point of $\Theta$ does not guaranteed that we have found a Kripkean truth model for it may be that $\theta^\xi(Y_f,f)\not\in\Theta^\xi(Y_f)\neq\emptyset$: $\theta^\xi(Y_f,f)$ may not be a fixed point of $\theta^\Phi$. In this case we would have $\theta(\theta^\xi(Y_f,f),Y_f)(J)=\emptyset$ for all $J\in X$ and
$$\{g\in \Theta^\xi(Y_f)\,|\,\theta^\xi(Y_f,f)\leq g\}=\{g\in \Theta^\xi(Y_f)\,|\,\theta^{\xi+1}(Y_f,f)\leq g\}$$ even though 
$$\theta^{\xi+1}(Y_f,f)\neq\theta^\xi(Y_f,f).$$
Summing up, in constructing the the fixed point of $\theta$ and $\Theta$ we need to make sure that the $\theta$ operation does not take us outside the grounded truth set $Y_f$, i.e., one needs to find fixed points $Z_g$ such that $\theta^\Phi(Z_g,g)\in\Theta^\Phi(Z_g)$. Then Lemma \ref{AFL} guarantees that we we have found an evaluation function $f$ and suitable grounded truth set $Y_f$ (with $f\in Y_f$) such that the two are fixed points of $\theta^\Phi$ and $\Theta^\Phi$ respectively.
 
By imposing suitable condition on the starting set $Y_f$ of the definition and the admissible condition $\Phi$ we can guarantee that the process of "sieving out" leads to a fixed point of $\Theta^\Phi$ with the desired properties. For example, Conditions \ref{c1}-\ref{c4} below amount to sufficient conditions on $Y_f$ (and $f$) and $\Phi$ that guarantee that none of the problems mentioned above can arise:

\begin{enumerate}[label=(\alph*)]
\item\label{c1} $Y_f$ is upward closed, i.e., if $f\leq g$ for $g\in\mathcal{B}_\mathfrak{M}$, then $g\in Y_f$;
\item $f\leq\theta(Y_f,f)$;
\item\label{c3} for all $g\in Y_f$, if $g\leq\theta(Y_f,g)\in Y_f$, then $\Phi(\theta(Y_f,g))\cap Y_f\neq\emptyset$;
\item\label{c4} $\Phi$ is compact on $Y_f$.
\end{enumerate}

With \ref{c1}-\ref{c3} in place we know that the fixed-point construction may only go wrong at limit stages and, indeed, for certain choices of admissibility conditions the construction will go wrong at limit ordinals as the following example illustrates.
\begin{exam}[$\omega$-consistency]\label{mcgomega}Let $\mathfrak{M}=(D,X,H)$ be a classical model, i.e., $X=\{J\}$ where $J$ is a classical interpretation and $H$ is $\{\langle J,J\rangle\}$. For $f\in{\sf Val}_\mathfrak{M}$ let $\Phi_\omega$ be defined as follows:
\begin{align*}\Phi_\omega(f)&:=\begin{cases}\emptyset,&\,{\rm if}\,f\not\in\mathcal{B}_\mathfrak{M}\\
\{g\in\mathcal{B}_\mathfrak{M}\,|\,f\leq g\,\&\,g(J)\text{ is $\omega$-consistent},&\,{\rm otherwise}.\end{cases}\end{align*}
Now let ${\sf f_{mc}}$ be a function such that for $n\in\omega$ formula $\varphi$
\begin{align*}{\sf f_{mc}}(n,\#\varphi):=\begin{cases}\#\varphi,&\text{ if }n=0;\\\#\T {\sf f}_{\sf mc}^\bullet(\overline{m},\gn\varphi),&\text{ if }n=m+1.\end{cases}\end{align*}
and where ${\sf f}_{\sf mc}^\bullet$ is a function symbol representing the p.r function ${\sf f_{mc}}$. Moreover, let $Y_f=\mathcal{B}_\mathfrak{M}$ whith $f(J)=\emptyset$ and let $\delta$ be the sentence $\neg\forall x\T {\sf f}_{\sf mc}^\bullet(x,\gn\delta)$. Then it is easy to see that
$$\T {\sf f}_{\sf mc}^\bullet(\overline{n},\gn\delta)\in\theta^\omega(Y_f,f)$$
 for all $n\in\omega$. Thus $\theta^\omega(Y_f,f)$ is $\omega$-inconsistent, which means that $\Theta^{\omega}(Y_f)=\emptyset$ and thus $\Phi_\omega(\theta^\omega(Y_f,f))\cap\Theta^\omega(Y_f)=\emptyset$. This implies that $\theta^{\omega+1}(Y_f,f)(J)={\sf Sent}$ and thus $\theta^{\omega+1}(Y_f,f)\not\in\mathcal{B}_\mathfrak{M}$.  See \cite{mcg91} and \cite{cam21} for further discussion.
\end{exam}

Problems like the one presented in Example \ref{mcgomega} can be blocked by Condition \ref{c4}, i.e., by requiring that the admissibility condition is compact.

\begin{de}[Compactness]Let $Y_f\in{\sf Adm}_\mathfrak{M}$ be a grounded truth set and $\Phi$ an admissibility condition. For $X\subseteq Y_f$ set $\Phi(X)=\{\Phi(f)\,|\,f\in X\}$. We say that $\Phi$ is compact on $Y_f$ iff for all $X\subseteq Y_f$:
\begin{align*}\text{if }\Phi(f_1)\cap\ldots\cap\Phi(f_n)\neq\emptyset\text{ for all $n\in\omega$ and $f_1,\ldots f_n\in X$},&\text{ then }\bigcap\Phi(X)\neq\emptyset.\end{align*}
\end{de}

If the admissibility condition is compact, that is, Condition \ref{c4} is met in addition to Conditions \ref{c1}-\ref{c3}, there will be fixed points of $\theta^\Phi$ and $\Theta^\Phi$.

\begin{prop}\label{GFP}Let $Y_f=\{g\in\mathcal{B}_\mathfrak{M}\,|\,f\leq g\}$ and let $\Phi$ be compact on $Y_f$. If $f\leq\theta^\Phi_\mathfrak{M}(Y_f,f)$ and for all $g\in Y_f$:
\begin{align*}g\leq\theta(Y_f,g)\in Y_f&\Rightarrow\Phi(\theta(Y_f,g))\cap Y_f\neq\emptyset,\end{align*}
then there exists $\xi$ such that for all $\alpha\geq\xi$
\begin{align*}\theta^\xi(Y_f,f)\in\Theta^\xi(Y_f)=\Theta^\alpha(Y_f).\end{align*}
\end{prop}

In the proof of the proposition we shall make use of the following auxiliary lemma.

\begin{lem}\label{gensub}Let $Y_f\in{\sf Adm}_\mathfrak{M}$ and for $g\in Y_f$ let $Y^{g\leq}_f:=\{h\in Y_f\,|\,g\leq h\}.$ Then  $$\text{if }\Phi(g)\cap Y_f\neq\emptyset\text{ then }\theta(Y_f,g)=\theta(Y^{g\leq}_f,g).\footnotemark$$
\end{lem}

Recall that by definition $f\leq h$ for all $h\in Y_f$. So, in particular, we know that $f\leq g$, which means that Lemma \ref{gensub} is a Lemma about $g$-generated $\leq$-substructures The lemma highlights that as in Kripke frames for ${\sf N3}$ truth at a world is {\it local} notion---only the worlds that are---potentially---accessible from a given world will be relevant.\footnote{We are not guaranteed that $f\leq_\Phi g$ for all $g\in Y_f$ such that $f\leq g$, so we actually still consider valuation functions that are not relevant for the evaluation of the conditional.}

\begin{proof}By induction on the complexity of $\varphi$ we show that
\begin{align*}(D,X\times Y_f,H_\Phi),J_g\Vdash\varphi&\text{ iff }(D,X\times Y^{g\leq}_f,H_\Phi),J_g\Vdash\varphi.\end{align*}
The only interesting case is the case of the conditional. The left-to-right direction is immediate since $Y^{g\leq}_f\subseteq Y^f$. Since $\Phi(g)\cap Y_f\neq\emptyset$  it suffices to observe that
$$\{h\in Y_f\,|\,g\leq_\Phi h\}\subseteq Y^{g\leq}_f$$
to establish the converse direction. 
\end{proof}

Ultimately, Lemma \ref{gensub} tells us that for suitable choices of the set $Y_f$ we can ignore the operation $\Theta^\Phi$ and rather than defining the minimal fixed point of $\theta^\Phi$ via a simultaneous inductive definition, inductively define  the fixed point in the parameter $Y_f$.

\begin{proof}[Proof of Proposition \ref{GFP}]
By induction on $\alpha$ we show that for all $\beta,\alpha$
$$\text{if }\beta<\alpha,\text{ then }\theta^\beta(Y_f,f)\leq \theta^\alpha(Y_f,f)\text{ and }\theta^\alpha(Y_f,f)\in\Theta^\alpha(Y_f).$$
As induction hypothesis we may assume that $\theta^\gamma(Y_f,f)\leq \theta^\beta(Y_f,f)$ for all $\beta,\gamma$ with $\gamma<\beta<\alpha$. If $\alpha=0$ there is nothing to show. Let $\alpha=\beta+1$. We conduct a secondary induction on $\beta$. If $\beta=0$ then $\alpha=1$ and the claim follows by assumption. Let $\beta=\gamma+1$. Then by IH of the primary induction  $\theta^\gamma(Y_f,f)\leq \theta^\beta(Y_f,f)$ and $\theta^\beta(Y_f,f)\in\Theta^\beta(Y_f)\subseteq Y_f$. Then by assumption we know that $\Phi(\theta^\beta(Y_f,f))\cap\Theta^\beta(Y_f)\neq\emptyset$ and $\Theta^\beta(Y_f)\subseteq\Theta^\gamma(Y_f)$. By Lemma \ref{gmon2} and Lemma \ref{gensub} this implies
$$\theta^\beta(Y_f,f)=\theta(\Theta^\gamma(Y_f),\theta^\gamma(Y_f,f))\leq\theta(\Theta^\beta(Y_f),\theta^\beta(Y_f,f))=\theta^\alpha(Y_f,f).$$
Moreover, by Lemma \ref{wdef} we also know that
$$\theta^\alpha(Y_f,f)\in\Theta^\alpha(Y_f)\in{\sf Adm}_\mathfrak{M}.$$

Let $\beta$ be a limit ordinal. We show that $\Phi(\theta^\beta(Y_f,f))\cap Y_f\neq\emptyset$. Notice that by the definition of $\Phi$ (for arbitrary $\Phi$) this shows that $\Phi(\theta^\beta(Y_f,f))\in Y_f$. Let $\Phi^\beta:=\{\Phi(\theta^\gamma(Y_f,f))\,|\,\gamma<\beta\}$. Then by IH $\bigcap X\neq\emptyset$ for every finite subset $X$ of $\Phi^\beta$ and by compactness
$$\bigcap\Phi^{\beta}\neq\emptyset.$$
It is then straightforward to verify that $\theta^\beta(Y_f,f)\leq\theta(\Theta^\beta(Y_f),\theta^\beta(Y_f,f)=\theta^\alpha(Y_f,f)\in\Theta^\alpha(Y_f)\in{\sf Adm}_\mathfrak{M}$.\footnote{Assume by reductio that this is not the case. Then there must be a sentence $\varphi\in\theta^\gamma(Y_f,f)(J)$ for some $\gamma<\beta$, interpretation $J$ such that $\varphi\not\in\theta(\Theta^\gamma(Y_f,f),\theta^\gamma(Y_f,f)(J)$. Contradiction.}

Finally, let $\alpha$ be a limit ordinal. Then the claim follows by definition of $\theta^\alpha(Y_f,f)$ and the fact that $\theta^\alpha(Y_f,f)\in Y_f$. 

Since the sequence of $\theta^\alpha$ for $\alpha\in{\sf ON}$ is progressive we may infer that there is an ordinal $\xi$ such that $\theta(\Theta^\xi(Y_f),\theta^\xi(Y_f,f))=\theta^\xi(Y_f,f)\in\Theta^\xi(Y_f)$, that is, $\theta^\xi(Y_f,f)=\theta^\alpha(Y_f,f)$ for all $\alpha\geq\xi$. This also implies that $\Theta(\Theta^\xi(Y_f))=\Theta^\xi(Y_f)$ and thus $\Theta^\xi(Y_f)=\Theta^\alpha(Y_f)$ for $\xi\leq\alpha$.
\end{proof}

Proposition \ref{GFP} is an abstract fixed-point result in the sense that it does not show that we can find grounded sets of valuation functions $Y_f$ and admissibility conditions $\Phi$ that satisfy conditions \ref{c1}-\ref{c4}. However, if there are, we know that fixed points exist and naive truth models alongside.

\begin{rem}Our proof of the abstract fixed-point result consisted in simultaneously inductively defining the interpretation of the truth predicate at an interpretation as well as the grounded set of valuation functions of the Kripkean truth structure. It is worth highlighting that as a consequence of Lemma \ref{gensub} it turns out that if the inductive definition is successful, we may take the initial set of grounded valuation functions $Y_f$ as a parameter in the inductive definition of the fixed point of $\theta$.

This observation also allows us to provide an algebraic perspective of the abstract fixed point result: instead of showing the existence of fixed points of $\theta^{\Phi_{\sf e}}_\mathfrak{M}$ by conducting an inductive definition, one can employ a variant of the Knaster-Tarski fixed-point result: to this effect one shows that $(Y^{+}_{f},\leq)$ with
$$Y^{+}_{f}:=\{g\in\mathcal{B}_\mathfrak{M}\,|\,f\leq g\,\&\,\Phi_{\sf e}(g)\cap\mathcal{B}_\mathfrak{M}\neq\emptyset\}$$
is a CCPO, i.e., a consistent and complete partial order \citep[cf.~][p.~556]{vis89}.\footnote{Note that the compactness of $\Phi_{\sf e}$ is crucial to this effect.} One then needs to show that the operation $\theta^{\Phi_{\sf e}}_ {Y^{+}_{f}}$ with
$$\theta^{\Phi_{\sf e}}_ {Y^{+}_{f}}(g):=\theta^{\Phi_{\sf e}}_\mathfrak{M}(Y^{+}_{f},g)$$
for all $g\in Y^{+}_{f}$ is monotone on $Y^{+}_{f}$, i.e., one needs to show that for all $g,h\in Y^{+}_{f}$:
\begin{enumerate}[label=(\roman*)]
\item\label{on} $\theta^{\Phi_{\sf e}}_ {Y^{+}_{f}}(g),\theta^{\Phi_{\sf e}}_ {Y^{+}_{f}}(h)\in Y^{+}_{f}$;
\item\label{valmon} if $g\leq h$, then $\theta^{\Phi_{\sf e}}_ {Y^{+}_{f}}(g)\leq\theta^{\Phi_{\sf e}}_ {Y^{+}_{f}}(h)$.
\end{enumerate}
The crucial point is to show \ref{on}, as \ref{valmon} then follows from Lemma \ref{gmon2}. Notice that it follows form the assumption in Proposition \ref{GFP} that 
\begin{align*}g\leq\theta(Y_f,g)\in Y_f&\Rightarrow\Phi(\theta(Y_f,g))\cap Y_f\neq\emptyset,\end{align*}
for all $g\in Y_f$. Once one has established that $(Y^{+}_{f},\leq)$ is a CCPO and that $\theta^{\Phi_{\sf e}}_ {Y^{+}_{f}}$ is monotone on $Y^{+}_f$ the existence of fixed points of $\theta^{\Phi_{\sf e}}_ {Y^{+}_{f}}$ follows by the generalisation of Knaster-Tarski discussed in, e.g., \cite{vis89}. Moreover, $({\sf Fix}^{\Phi_{\sf e}}_{Y^{+}_{f}},\leq)$ is itself a CCPO with
$${\sf Fix}^{\Phi_{\sf e}}_{Y^{+}_{f}}:=\{g\in Y^{+}_{f}\,|\,\theta^{\Phi_{\sf e}}_ {Y^{+}_{f}}(g)=g\}.$$
We refer to \cite[p.~556-664]{vis89} for a general discussion of algebraic fixed-fixed point constructions.
\end{rem}

\subsection{Applications of Proposition \ref{GFP}}
We now wish to provide fixed-point results for specific admissibility condition. We start by introducing different admissibility conditions. \begin{de}Let $f\in\mathcal{B}_\mathfrak{M}$ and let $L$ be some logic. Then $f$ is $L$-saturated iff for all $J\in X$, all $\Delta\subseteq\lt$ all $\Gamma\subset f(J)$ such that
\begin{align*}{\rm if}\,\vdash_L\Gamma\Rightarrow\Delta,&\,\,{\rm then}\,\,\Delta\cap f(J)\neq\emptyset.\end{align*}

Moreover for $f\in{\sf Val}_\mathfrak{M}$ we set:
\begin{align*}
&&\Phi_{\sf c}(f) &:=\begin{cases}\emptyset,&\text{ if }f\not\in\mathcal{B}_\mathfrak{M};\\\{g\in\mathcal{B}_\mathfrak{M}\,|\,f\leq g\},&\text{ otherwise.}\end{cases}\\
&&\Phi_{\sf K3}(f)&:=\begin{cases}\emptyset,&\text{ if }f\not\in\mathcal{B}_\mathfrak{M};\\\{g\in\mathcal{B}_\mathfrak{M}\,|\,f\leq g\,\&\,g\text{ is {\sf K3}-saturated}\},&\text{ otherwise.}\end{cases}\\
&&\Phi_{\sf N3}(f) &:=\begin{cases}\emptyset,&\text{ if }f\not\in\mathcal{B}_\mathfrak{M};\\\{g\in\mathcal{B}_\mathfrak{M}\,|\,f\leq g\,\&\,g\text{ is {\sf N3}-saturated}\},&\text{ otherwise.}\end{cases}\\
&&\Phi_{\sf Nve}(f) &:=\begin{cases}\emptyset,&\text{ if }f\not\in\mathcal{B}_\mathfrak{M};\\\{g\in\mathcal{B}_\mathfrak{M}\,|\,f\leq g\,\&\,g\text{ is ${\sf K3\T}$-saturated}\},&\text{ otherwise.}\end{cases}
\end{align*}
\end{de}
The different logics are displayed in form of their sequent calculi in Appendix \ref{SeqC}. The logic ${\sf K3\T}$ is the naive logic of truth in the sense of, e.g., \cite{kre88}, i.e., the logic obtained from ${\sf K3}$ by the addition of naive truth rules.

The next step would be to show the existence of fixed points of the operation $\theta_{\sf e}$ and $\Theta_{\sf e}$ for ${\sf e}\in\{{\sf c},{\sf K3},{\sf N3},{\sf Nve}\}$, that is, to show that there is an $f\in\mathcal{B}_\mathfrak{M}$ and $Y_f:=\{g\leq f\,|\,g\in\mathcal{B}_\mathfrak{M}\}\in{\sf Adm}_\mathfrak{M}$ such that there exists $g\in Z_g\subseteq Y_f$ with $\theta^{\Phi_{\sf e}}_\mathfrak{M}(Z_g,g)=g$ and $\Theta^{\Phi_{\sf e}}_\mathfrak{M}(Z_g)=Z_g$. Indeed, if we let $f$ be the valuation function $f(J)=\emptyset$ for all $J\in X$, then one can show that $f$ and $Y_f$ satisfy all the conditions for applying Proposition \ref{GFP} for ${\sf e}\in\{{\sf c},{\sf K3},{\sf N3},{\sf Nve}\}$. However, we directly prove a stronger fixed-point result, that is, we show that we can fixed points $Z_g$ of $\Theta^{\Phi_{\sf e}}$ such that $h(J)$ is a fixed point of the conditional free fragment of $\lt$ ($\lt^-$)  for every $h\in Z_g$.
To this effect we let $\mathcal{B}^{\T}_\mathfrak{M}\subseteq\mathcal{B}_\mathfrak{M}$ be the set ov valuation functions such that for all $f\in \mathcal{B}^{\T}_\mathfrak{M}$ and all $J\in X$:
\begin{enumerate}[label=(\roman*)]
\item $f(J)$ consistently extends a ${\sf K3}$-fixed point $S$, i.e., $f(J)\cap\lt^-$ is ${\sf K3}\T$-saturated;
\item $f(J)$ is ${\sf K3}$-saturated;
\end{enumerate}

Our aim is to find a set $Y_f^{\T}\subseteq\mathcal{B}^{\T}_\mathfrak{M}$ such that for ${\sf e}\in\{{\sf K3},{\sf N3},{\sf Nve}\}$ (notice that $\Phi_{\sf c}$ and $\Phi_{\sf K3}$ coincide over $\mathcal{B}^{\T}_\mathfrak{M}$):
\begin{itemize}
\item $\Phi_{\sf e}$ is compact on $Y_f^{\T}$;
\item $f\leq\theta^{\Phi_{\sf e}}(Y^{\T}_f,f)\in Y_f^{\T}$;
\item for all $g\in Y^{\T}_f$:
$$\text{ if }g\leq\theta^{\Phi_{\sf e}}(Y^{\T}_f,g)\in Y^{\T}_f,\text{ then }\Phi_{\sf e}(\theta^{\Phi_{\sf e}}(Y^{\T}_f,g))\cap Y_f\neq\emptyset.$$
\end{itemize}
Once we find such a set, the existence of fixed points follows from Proposition \ref{GFP}.

Let $\mathcal{K}:\mathcal{P}({\sf Sent})\rightarrow\mathcal{P}({\sf Sent})$ be an operation such that for all $J\in X$, $g\in{\sf Val}_\mathfrak{M}$ and all $\varphi\in\lt$:
\begin{align*}
\varphi\in\mathcal{K}(g(J))&\text{ iff }\begin{cases}
J(t)=J(s),&\text{ if }\varphi\doteq(t=s);\\
J(t)\neq J(s),&\text{ if }\varphi\doteq(\neg t=s);\\
\langle J(t_1),\ldots,J(t_n)\rangle\in J^+(P^n),&\text{ if }\varphi\doteq(P t_1,\ldots,t_n);\\
\langle J(t_1),\ldots,J(t_n)\rangle\in J^-(P^n),&\text{ if }\varphi\doteq(\neg P t_1,\ldots,t_n);\\
\psi\in g(J), &\text{ if }\varphi\doteq\T t\,\&\,J(t)=\psi;\\
\neg\psi\in g(J), &\text{ if }\varphi\doteq\neg \T t\,\&\,J(t)=\psi;\\
\psi\in g(J)\,\&\,\chi\in g(J),&\text{ if }\varphi\doteq(\psi\wedge\chi);\\
\neg\psi\in gl(J)\text{ or }\neg\chi\in g(J),&\text{ if }\varphi\doteq\neg (\psi\wedge\chi);\\
\psi(c)\in g(J)\text{ for all }c\in{\sf Const}_{\lt},&\text{ if }\varphi\doteq\forall x\psi;\\
\psi(c)\in g(J)\text{ for some }c\in{\sf Const}_{\lt},&\text{ if }\varphi\doteq\neg\forall x\psi;
\end{cases}\end{align*}
Since $\mathcal{K}$ is monotone on $\mathcal{B}$, and $\mathcal{K}(g(J))$ is consistent for all $J\in X$ and $g\in\mathcal{B}$, we know that there is an $f_{\sf k}\in\mathcal{B}$ such that $f_{\sf k}(J)$ is the minimal $\mathcal{K}$ fixed point over the set $\{\varphi\,|\,(D,J)\Vdash_{\sf K3}\varphi\}$ for all $J\in X$. Moreover, by the properties of $\mathcal{K}$ clearly $f_{\sf k}\in\mathcal{B}^{\T}_\mathfrak{M}$.

\begin{lem}\label{MonT}Let ${\sf e}\in\{{\sf c},{\sf K3},{\sf N3},{\sf Nve}\}$ and $Y^{\T}_{f_{\sf k}}:=\{g\in\mathcal{B}^{\T}_\mathfrak{M}\,|\,f_{\sf k}\leq g\}$ with $f_{\sf k}\in\mathcal{B}^{\T}_\mathfrak{M}$ as discussed above. Then
\begin{enumerate}[label=(\alph*)]
\item $f_{\sf k}\leq\theta^{\Phi_{\sf e}}(Y^{\T}_{f_{\sf k}},f_{\sf k})\in Y_{f_{\sf k}}^{\T}$;
\item for all $g\in Y^{\T}_f$,  if $g\leq\theta^{\Phi_{\sf e}}(Y^{\T}_f,g)\in Y^{\T}_f$, then $\Phi_{\sf e}(\theta^{\Phi_{\sf e}}(Y^{\T}_f,g))\cap Y_f\neq\emptyset.$ 
\end{enumerate}
\end{lem}

\begin{proof}
(a) is straightforward to verify. Notice that $\theta_{\Phi_{\sf e}}( Z_g,g)$ is ${\sf N3}$-saturated for all $g\in Z_g\in{\sf Adm}_\mathfrak{M}$ for ${\sf e}\in\{{\sf c},{\sf K3},{\sf N3},{\sf Nve}\}$, and that for $\varphi\in\lt^{-}$:
\begin{align*}\varphi\in f_{\sf k}(J)&\,{\rm iff }\,\varphi\in\theta^{\Phi_{\sf e}}(Y^{\T}_{f_{\sf k}},f_{\sf k}).\end{align*}

For (b) notice that for ${\sf e}\in\{{\sf c},{\sf K3},{\sf N3}\}$, the claim is immediate since $\theta^{\Phi_{\sf e}}_\mathfrak{M}(Y_{f_{\sf k}}^{\T},g)\in\Phi(\theta^{\Phi_{\sf e}}(Y^{\T}_{f_{\sf k}},g))$. For ${\sf e}={\sf Nve}$ we observe that $\theta^{\Phi_{\sf e}}_\mathfrak{M}(Y_{f_{\sf k}}^{\T},g)(J)\subseteq\mathcal{K}([\theta^{\Phi_{\sf e}}_\mathfrak{M}(Y_{f_{\sf k}}^{\T},g)(J)])$ for all $J\in X$. Moreover, since by assumption $\theta^{\Phi_{\sf e}}_\mathfrak{M}(Y_{f_{\sf k}}^{\T},g)(J)$ is consistent if follows that $\mathcal{K}([\theta^{\Phi_{\sf e}}_\mathfrak{M}(Y_{f_{\sf k}}^{\T},g)(J)])$ is consistent. This means that there is a consistent set $S_J\supseteq\theta^{\Phi_{\sf e}}_\mathfrak{M}(Y_{f_{\sf k}}^{\T},g)(J)$ such that $\mathcal{K}(S_J)=S_J$. Define a valuation function $g'$ such that for all $I\in X$:
$$g'(I):=S_I.$$ 
Then we have $g'\geq \theta^{\Phi_{\sf e}}_\mathfrak{M}(Y_{f_{\sf k}}^{\T},g)$ and $g'$ is ${\sf K3}\T$-saturated and $g'\in Y^{\T}_{f_{\sf k}}$: $$\Phi_{\sf e}(\theta^{\Phi_{\sf e}}(Y^{\T}_{f_{\sf k}},g))\cap Y^{\T}_{\sf k}\neq\emptyset.$$

\end{proof}

It remains to show that $\Phi_{\sf e}$ is compact on $Y^{\T}_{\mathfrak{M}}$ before we can apply Proposition \ref{GFP}.

\begin{lem}\label{CPL}Let $Y^{\T}_{f_{\sf k}}$ be defined as in Lemma \ref{MonT} and ${\sf e}\in\{{\sf c},{\sf K3},{\sf N3},{\sf Nve}\}$. Then $\Phi_{\sf e}$ is compact on $Y^{\T}_{f_{\sf k}}$.
\end{lem}
\begin{proof}Let $Z$ be an arbitrary subset of $Y^{\T}_{f_{\sf k}}$ and assume that $\Phi_{\sf e}(f_1)\cap\ldots\cap\Phi_{\sf e}(f_n)\neq\emptyset$ for all $n\in\omega$ and $f_1,\ldots,f_n\in Z$.


Now let ${\sf e}\in\{{\sf K3,N3,K3\T}\}$. We show that $\bigcup\{f(J)\,|\,f\in Z\}$ is ${\sf e}$-consistent. Assume for reductio that $\bigcup\{f(J)\,|\,f\in Z\}$ is not ${\sf e}$-consistent for some $J\in X$, that is, there is a finite set of formulas $\Gamma\subseteq\bigcup\{f(J)\,|\,f\in Z\}$ such that
$\Gamma\vdash_{\sf e}\bot$. But then there must be an $f_1,\ldots,f_n\in Z$ for some $n\in\omega$ such that $\Gamma\subseteq f_1(J)\cup\ldots\cup f_n(J)$. This implies that $\Phi_{\sf e}(f_1)\cap\ldots\cap\Phi_{\sf e}(f_n)=\emptyset$: contradiction. Moreover, since every $f(J)$ is ${\sf K3}$-saturated and $\bigcup\{f(J)\,|\,f\in Z\}$ is ${\sf K3}$-saturated and by essentially the same reasoning $\bigcup\{f(J)\,|\,f\in Z\}\cap\lt^{-}$ is ${\sf K3}\T$ saturated. This guarantees that we can extend $\bigcup\{f(J)\,|\,f\in Z\}$ to an ${\sf e}$-saturated set $g(J)$ for ${\sf e}\in\{{\sf N3},{\sf K3}\T\}$ such that $g(J)\cap\lt^{-}$ is ${\sf K3}\T$ saturated and then define an admissible valuation function $g$ accordingly. By construction $g\in\bigcap_{f\in Z}\Phi(f)$.
\end{proof}
We have now gathered all relevant facts to establish the existence of fixed points of $\theta^{\Phi_{\sf e}}$ and $\Theta^{\Phi_{\sf e}}$ for ${\sf e}\in\{{\sf c},{\sf K3},{\sf N3},{\sf Nve}\}$.
\begin{prop}\label{CFPT}Let $\mathfrak{M}$ be a supervaluation structure and let $Y^{\T}_{f_{\sf k}}$ be defined as in Lemma \ref{MonT} and ${\sf e}\in\{{\sf c},{\sf K3},{\sf N3},{\sf Nve}\}$. Then there exists $g\in Z_g\subseteq Y^{\T}_f$ such that $\theta^{\Phi_{\sf e}}_\mathfrak{M}(Z_g,g)=g$ and $\Theta^{\Phi_{\sf e}}_\mathfrak{M}(Z_g)=Z_g$.
\end{prop}
\begin{proof} By Proposition \ref{GFP}, and Lemmas \ref{MonT} and \ref{CFPT}.\end{proof}

\begin{cor}\label{CFPTM}Let $\mathfrak{M}=(D,X,H)$ be a supervaluation structure and ${\sf e}\in\{{\sf c},{\sf K3},{\sf N3},{\sf Nve}\}$.  Then there exists a grounded truth set $Y_f\subseteq \mathcal{B}^{\T}_\mathfrak{M}$ and admissible valuation function $f$ such that for all $\varphi\in{\sf Sent}_{\lt}$
\begin{align*}(D,X\times Y_f,H_{\Phi_{\sf e}}),J_f\Vdash\varphi&\text{ iff }(D,X\times Y_f,H_{\Phi_{\sf e}}),J_f\Vdash\T\gn\varphi\end{align*}
for all $J\in X$.\end{cor}

Corollary \ref{CFPTM} establishes the existence of Kripkean truth structures. However, strictly speaking they do not provide us with supervaluation structures for the language $\lt$ that comprise Kripkean truth models. The reason is that a relation $H$ on $X$ was deemed an admissibility relation only if it is reflexive. For arbitrary Kripkean truth structures we are not guaranteed that $H_{\Phi_{\sf e}}$ is reflexive on $X\times Y_f$ for ${\sf e}\in\{{\sf K3, N3, Nve}\}$ as $Y_f$ may contain valuation functions $g$ such that $g(J)$ is not ${\sf e}$-saturated and thus $(J_g,J_g)\not\in H_{\Phi_{\sf e}}$. Yet, it is straightforward to obtain a Kripkean supervaluation structure for $\lt$ from a Kripkean truth structure $(D,X\times Y_f,H_{\Phi_{\sf e}})$ by forming the $f$-generated $\leq_{\Phi_{\sf e}}$-substructure.

\begin{cor}[Kripkean supervaluation structure]\label{KSS}Let $(D,X\times Y_f,H_{\Phi_{\sf e}})$ be a Kripkean truth structure for $\lt$ and ${\sf e}\in\{{\sf K3, N3, Nve}\}$. Set $$Y^{\leq_{\Phi_{\sf e}}}_f:=\{g\in Y_f\,|\,f\leq_{\Phi_{\sf e}}g\}.$$
Then $(D,X\times Y^{\leq_{\Phi_{\sf e}}}_f,H_{\Phi_{\sf e}})$ is a Kripkean truth structure such that $H_{\Phi_{\sf e}}$ is reflexive on $X\times Y^{\leq_{\Phi_{\sf e}}}_f$ and $\theta^{\Phi_{\sf e}}_\mathfrak{M}(Y^{\leq_{\Phi_{\sf e}}}_f,f)=f$.\end{cor}

Before we turn to the properties of the conditional in Kripkean Truth Structures we reflect on some constraints and limitations of the framework of the construction. In particular, we reflect again on the idea of transparent truth.

\begin{rem}[Transparency]\label{Transp}One may ask why the $\Phi_{\sf Nve}$ construction does not lead to transparent truth models (in contradiction to Curry's paradox). After all, the conditional gets, in this case, evaluated with respect to models with naive valuations functions. The short answer is that a naive valuation does not guarantee that the model is a naive truth model.
Indeed, for every $\Phi_{\sf Nve}$-fixed point $(Y_f,f)$ there must be a $g\in\Phi_{\sf Nve}(f)\cap Y_f$ such that $(\mathfrak{M}_{\T},J_g)$ is not a naive truth model for all $J\in X$. This can be illustrated by looking at the Curry sentence $\kappa$ for it is easy to see that if $\kappa\in g(J)$ for $J\in X$, then $(\mathfrak{M}_{\T},J_g)$ is not a naive truth model. Suppose otherwise. Then since by assumption $(\mathfrak{M}_{\T},J_g)\Vdash\T\gn\kappa$ and by naivity
\begin{align*}(\mathfrak{M}_{\T},J_g)\Vdash\T\gn\kappa&\,{\rm iff}\,(\mathfrak{M}_{\T},J_g)\Vdash\kappa\end{align*}
it follows that $(\mathfrak{M}_{\T},J_g)\Vdash\kappa$.  By definition of $\Phi_{\sf Nve}$ we know that $g\in\Phi_{\sf Nve}(g)$ and thus, by Definition \ref{TSV}, we obtain
\begin{align*}{\rm if }\,(\mathfrak{M}_{\T},J_g)\Vdash\T\gn\kappa,&\,{\rm then }\,(\mathfrak{M}_{\T},J_g)\Vdash\bot.\end{align*}
This yields a contradiction and we infer that if $\kappa\in g(J)$ for some naive, non-trivial valuation, then $(\mathfrak{M},J_g)$ is not a naive truth model. Accordingly, if $f$ is a $\Phi_{\sf Nve}$-fixed point, then $\kappa\not\in f(J)$ for $J\in X$.  Also, since $f$ is a fixed-point this implies by naivity that $(\mathfrak{M}_{\T},J_f)\not\Vdash\kappa$ and by Definition \ref{TSV}
$$\exists J'_{g}((J_f,J'_g)\in H_{\Phi_{\sf Nve}}\,\&\,(\mathfrak{M}_{\T},J'_g)\Vdash\T\gn\kappa).$$
The latter implies that $\kappa\in g(J')$ and thus $(\mathfrak{M}_{\T},J'_g)$ is not a naive truth model, which proves our initial claim. It is then straightforward to see that $f$ is not transparent since $\kappa\rightarrow\kappa\in f(J)$ for all $J\in X$, but $\T\kappa\rightarrow\kappa\not\in f(J)$, as  $(\mathfrak{M}_{\T},J'_g)\Vdash\T\gn\kappa)$ and  $(\mathfrak{M}_{\T},J'_g)\not\Vdash\kappa)$.
\end{rem}

In the above remark we highlighted that the existence of naive truth models requires for $J\in X$ the existence of admissible precisifications $f$ such that $\kappa\in f(J)$. This in turn shows that we cannot find non-trivial fixed points if the following admissibility condition is assumed:
\begin{align*}
\Phi_{\sf N3-Nve}(f)&:=\begin{cases}\emptyset,&\text{ if }f\not\in\mathcal{B}_\mathfrak{M};\\\{g\in{\sf Val}^\mathcal{B}_\mathfrak{M}\,|\,f\leq g\,\&\,g\text{ is ${\sf N3}$-naive}\},&\text{ otherwise.}\end{cases}
 \end{align*}

\begin{lem}\label{nn3nve}Let $\mathfrak{M}=(D,X,H)$ be a supervaluation structure. Then there is no $Y\subseteq\mathcal{B}_\mathfrak{M}$ and $f\in\mathcal{B}_\mathfrak{M}$ such that $f\leq\theta^{\Phi_{\sf N3-Nve}}_\mathfrak{M}(Y,f)\in\mathcal{B}_\mathfrak{M}$.\end{lem}

\begin{proof}Let $J \in X$. Then there are two cases: either (i) $\kappa\in f(J)$ or (ii) $\kappa\not\in f(J)$. In case (i) we know that $\kappa\not\in[\theta^{\Phi_{\sf N3-Nve}}_\mathfrak{M}(f)](J)$ and thus $f\not\leq\theta^{\Phi_{\sf N3-Nve}}_\mathfrak{M}(f)$. (ii) Assume then that $\kappa\not\in f(J)$. Notice that there is no ${\sf N3-Nve}$ valuation function $g$ such that $\kappa\in g(J)$ for some $J\in X$. Assume otherwise. Then by naivity $\T\gn\kappa\in g(J)$ and by modus pones $\bot\in g(J)$. The latter implies that $g\not\in\mathcal{B}_\mathfrak{M}$. But since for all $g\in\Phi_{\sf N3-Nve}(f)$ we have $\kappa\not\in g(J)$ for $J\in X$ it follows $\kappa\in[\theta^{\Phi_{\sf N3-Nve}}_\mathfrak{M}(f)](J)$ and $[\theta^{\Phi_{\sf N3-Nve}}_\mathfrak{M}(f)](J)\not\in\mathcal{B}_\mathfrak{M}$ since by definition $[\theta^{\Phi_{\sf N3-Nve}}_\mathfrak{M}(f)](J)$ is ${\sf N3}$-closed.
\end{proof}

The moral of Lemma \ref{nn3nve} is that fixed points are only to be found iff valuations conceive of the conditional as an unanalysed unit,\footnote{At least for those conditionals that cannot be decided in ${\sf K3}$} or conceive of truth as a non-naive predicate. The Curry paradox forces us to opt for one of these options.

 \section{The truth predicate and the conditional in Strong Kleene Supervaluation}\label{truththeo}
The properties of $\T$ and $\rightarrow$ on a Kripkean supervaluation structure $\mathfrak{M}_{\T}=(D,X\times Y_f,H_{\Phi_{\sf e}}$ will obviously depend on the admissibility condition $\Phi_{\sf e}$ employed. Before we discuss the differences between the different admissibility conditions, we investigate the rules and principles that turn out true on all admissibility conditions we discuss. However, building on Proposition \ref{CFPT} we assume that $Y_f\subseteq\mathcal{B}_{\T}$, i.e., that the valuations function of the structure are Kripkean fixed points for $\lt^-$ and ${\sf K3}$-saturated. This will guarantee that the $\T$-scheme will hold for $\lt^-$-sentences. Moreover, by Corollary \ref{KSS} we may assume $H_{\Phi_{\sf e}}$ to be reflexive on $X\times Y_f$. This, in turn, guarantees that all the logical truths of Nelson logic ${\sf N3}$ will be true at any truth interpretation of the structure. In particular, {\it modus ponens} will hold for all $J_g\in X\times Y_f$ if $\mathfrak{M}_{\T},J_g\Vdash\varphi$ and $\mathfrak{M}_{\T},J_g\Vdash\varphi\rightarrow\psi$, then $\mathfrak{M}_{\T},J_g\Vdash\psi$. 
In contrast, the deduction theorem will only hold globally, that is, we have:
\begin{align*}\Gamma,\varphi\vDash_{\mathfrak{M}_{\T}}\psi&\,\,{\rm iff }\,\,\Gamma\vDash_{\mathfrak{M}_{\T}}\varphi\rightarrow\psi.\footnotemark\end{align*}
\footnotetext{By $\Gamma\vDash_{\mathfrak{M}_{\T}}\Delta$ we denote that truth is preserved from $\Gamma$ to $\Delta$ in $\mathfrak{M}_{\T}$ relative to all truth interpretations $J_g\in X\times Y_f$.}

Let us now turn to the properties  of the truth predicate in the Kripkean truth structures and collect some important facts:
\begin{enumerate}[label=(\roman*)]
\item for all $\varphi\in\lt$
\begin{align*}\mathfrak{M}_{\T},J_f\Vdash\T\gn\varphi\,\,&{\rm iff}\,\,\mathfrak{M}_{\T},J_f\Vdash\varphi\\
\mathfrak{M}_{\T},J_f\Vdash\neg\T\gn\varphi\,\,&{\rm iff}\,\,\mathfrak{M}_{\T},J_f\Vdash\neg\varphi\end{align*}
\item for all $\varphi\in\lt$, if $\mathfrak{M}_{\T}\Vdash\varphi\vee\neg\varphi$ or if $\varphi\in\mathcal{L}\cup\lt^{-}$, then
\begin{align*}&\mathfrak{M}_{\T}\Vdash\varphi\,\,{\rm iff}\,\,\mathfrak{M}_{\T}\Vdash\T\gn\varphi\\
&\mathfrak{M}_{\T}\Vdash\neg\varphi\,\,{\rm iff}\,\,\mathfrak{M}_{\T}\Vdash\neg\T\gn\varphi\\
&\mathfrak{M}_{\T}\Vdash\varphi\leftrightarrow\T\gn\varphi\\
&\mathfrak{M}_{\T}\Vdash\neg\varphi\leftrightarrow\neg\T\gn\varphi
\end{align*}
\item Let $\Gamma,\Delta\subseteq\lt$ and $\Gamma\vDash_{\sf N3}\Delta$ or $\Gamma\vDash_{\mathfrak{M}_{\T}}\Delta$. Then:
\begin{align*}{\rm if }\,\mathfrak{M}_{\T},J_f\Vdash\T\gn\gamma\text{ for all }\gamma\in\Gamma,&\,{\rm then }\,\mathfrak{M}_{\T},J_f\Vdash\T\gn\delta\text{ for some }\delta\in\Delta.\end{align*}
\end{enumerate}
(i) simply states that $J_f$ is a Kripkean fixed-point on $\mathfrak{M}_{\T}$. (ii) tells us that the $\T$-scheme holds on $\mathfrak{M}_{\T}$ on the $\mathfrak{M}_{\T}$-determinate/stable fragment of $\lt$, while (iii) tells us that the minimal fixed point is closed under ${\sf N3}$-consequences.

Notice that (ii) provides us with some insights about the behaviour of $\lt^-$ pathological sentence (in Kripke's sense) such as the Liar sentence $\lambda$. By the diagonal lemma we have $\mathfrak{M}_{\T}\Vdash\lambda\leftrightarrow\neg\T\gn\lambda$, but also by construction:
\begin{align*}&\mathfrak{M}_{\T}\Vdash\lambda\leftrightarrow\T\gn\lambda;\\
&\mathfrak{M}_{\T}\Vdash\lambda\leftrightarrow\neg\lambda;\\&\mathfrak{M}_{\T}\Vdash\T\gn\lambda\leftrightarrow\neg\T\gn\lambda;\\
&\mathfrak{M}_{\T}\Vdash\lambda\leftrightarrow\bot.\end{align*}
Of course, in classical logic this would lead to a contradiction, but not so in our framework: in our case this simply reflects the fact that neither $\lambda$ nor $\neg\lambda$ can be a member of a consistent $\lt^-$-fixed point.

Thus far we have focused on general rules and principles that are valid in truth structures and the Kripkean truth models independently of the admissibility condition assumed. However, depending on the admissibility condition assumed different truth principles will come out as true on the structure $\mathfrak{M}_{\T}$, as these truth principles can be seen as expressing particular closure conditions for the interpretations of the truth predicate. For ease of presentation we take the language of ${\sf PA}$, $\mathcal{L}_{\sf PA}$, to be the language of our syntax theory, that is, we handle syntax via arithmetization. We assume an extension of the usual signature $\{{\sf 0},{\sf S},{\sf +},{\sf \times}\}$ by finitely many function symbols for specific primitive recursive functions and $\mathcal{L}^{\rightarrow}_{\sf PA}$ the extension of $\mathcal{L}_{\sf PA}$ with the conditional connective. We adopt Feferman's \citeyearpar{fef60} dot-notation in representing syntax and let the set of $\lt$-sentence, formulas, and terms represent themselves. Importantly, the function symbol $\dot{\eta}$ (with $\eta$ being a placeholder for the argument of the function) represents the function which takes G\"odel numbers of expressions as arguments and outputs the G\"odel number of the numeral of the G\"odel number. With these preliminaries out of the way we focus on the following truth principles: 
\begin{enumerate}[label=(\alph*)]
\item\label{Pone} $\forall x\left({\sf Sent}(x)\rightarrow(\T x\wedge\T \subdot{\neg x}\rightarrow\bot)\right)$;
\item\label{Ptwo}$\forall x\left({\sf Sent}(x)\rightarrow (\T\subdot{\neg}x\leftrightarrow\neg\T x)\right)$;
\item\label{Pthree}$\forall x\left({\sf Sent}(x)\rightarrow (\T\subdot{\neg}\subdot{\neg}x\leftrightarrow\T x)\right)$;
\item\label{Pfour}$\forall x\forall y\left({\sf Sent}(x)\wedge{\sf Sent}(y)\rightarrow(\T(x\subdot{\wedge}y)\leftrightarrow\T x\wedge \T y)\right)$;
\item\label{Pfive}$\forall x\forall y\left({\sf Sent}(x)\wedge{\sf Sent}(y)\rightarrow(\T\subdot{\neg}(x\subdot{\wedge}y)\leftrightarrow\T\subdot{\neg}x\vee \T\subdot{\neg}y)\right)$;
\item\label{Psix}$\forall x\forall v\left({\sf Frml}^{1}(x)\wedge{\sf Var}(v)\rightarrow(\T\subdot{\forall}vx\rightarrow\forall y\T x(\dot{y}/v))\right)$;
\item\label{Pseven}$\forall x\forall v\left({\sf Frml}^{1}(x)\wedge{\sf Var}(v)\rightarrow(\neg\forall y\T x(\dot{y}/v)\rightarrow\T\subdot{\neg}\subdot{\forall}vx)\right)$;
\item\label{Peight}$\forall x\forall y\left({\sf Sent}(x)\wedge{\sf Sent}(y)\rightarrow( \T(x\subdot{\rightarrow}y)\wedge \T x\rightarrow\T y)\right)$;
\item\label{Pnine}$\forall x\forall y\left({\sf Sent}(x)\wedge{\sf Sent}(y)\rightarrow (\T\subdot{\neg}(x\subdot{\rightarrow}y)\leftrightarrow \T x\wedge\T\subdot{\neg} y)\right)$;
\item\label{Pten}$\forall x\left({\sf Sent}(x)\rightarrow (\T\subdot{\T}\dot{x}\leftrightarrow \T x)\right)$;
\item\label{Peleven}$\forall x\left({\sf Sent}(x)\rightarrow (\T\subdot{\neg}\subdot{\T}\dot{x}\leftrightarrow\T\subdot{\neg} x)\right)$.
\end{enumerate}

Principle \ref{Pone} says that truth predicate is consistent, \ref{Ptwo} that a sentence is false iff it is not true, while \ref{Pthree} to \ref{Pseven} assert that the truth predicate is closed under strong Kleene logic. Principles \ref{Peight} and \ref{Pnine} state that the truth predicate is closed under modes ponens and the compositional rules for negated conditionals. \ref{Pten} and \ref{Peleven} tell us that interpretation of the truth predicate is closed under the ${\sf K3}\T$ truth-introduction rules. We arrive at the following picture:

\begin{lem}\label{princ}Let $\mathfrak{M}_{\T}=(D,X\times Y_f,H_{\Phi_{\sf e}})$ be a Kripkean truth structure. Then the following principles are true in $\mathfrak{M}_{\T}$:
\begin{itemize}
\item \ref{Pone}-\ref{Pseven}, if ${\sf e}\in\{{\sf K3},{\sf N3},{\sf Nve}\}$;
\item \ref{Pone}-\ref{Pnine}, if ${\sf e}={\sf N3}$;
\item \ref{Pone}-\ref{Pseven}, \ref{Pten} and \ref{Peleven}, if ${\sf e}={\sf Nve}$;
\end{itemize}
\end{lem}

Lemma \ref{princ} is an immediate consequence of the properties forced on the interpretation of the truth predicate by the various admissibility conditions. Interestingly, it highlights again the fact that the interplay of the truth predicate and the conditional always requires us to choose between a the logicality of the conditional vs the logicality of truth: as an immediate consequence of Curry's paradox, if both \ref{Peight} and \ref{Pten} hold for a given structure $\mathfrak{M}_{\T}$, then $\mathfrak{M}_{\T}$ cannot contain Kripkean truth models.
 
Attentive readers may note the absence of the converse directions of Principles \ref{Psix} and \ref{Pseven}. The converse direction of \ref{Psix} would require all truth interpretations to be $\omega$-complete, whereas the converse direction of \ref{Pseven} requires interpretations to have witnesses for every existential statement.\footnote{These two properties coincide in classical logic for maximally consistent sets of sentences,  but they do not coincide for saturated sets of non-classical logics.}
Enforcing these properties onto to the interpretations of the truth predicate leads to non-compact admissibility conditions, which means that Proposition \ref{GFP} will no longer be applicable. However, in stark contrast to the classical case (cf.~Example \ref{mcgomega}) we can still find fixed points, if we adopt admissibility conditions that require admissible interpretation to be $\omega$ ${\sf e}$-saturated, that is, interpretation $\omega$-complete saturated sets that have the Henkin property. We leave the discussion of such non-compact admissibility conditions to a sequel paper in which we also develop and investigate axiomatic theories of truth in three-valued Nelson logic ${\sf N3}$. 
 
 
\section{Complexity}\label{complex}We now turn to investigating the recursion-theoretic complexity of the Kripkean fixed points over a structure $\mathfrak{M}=(D,X,H)$. Recall that, as remarked at the end of Section \ref{KTS} we can conceive of our construction as an inductive definition in the parameter $Y_f$ rather than a simultaneous inductive definition of a fixed point of $\theta$ and $\Theta$. We now show that for every $\theta_\mathfrak{M}$-fixed point $(Y_f,f)$ and all $J\in X$ the set $f(J)$ is $\Pi^1_1$-hard in the parameters $Y_f$, $\mathfrak{M}$, and $\Phi$. As a corollary this also tells us that the minimal fixed point of $\theta_\mathfrak{M}$ will be $\Pi^1_1$-complete if the relevant parameters are arithmetically definable.

When investigating the recursion-theoretic complexity it is convenient to think of $\mathcal{L}$ as an extension of the language of arithmetic $\mathcal{L}_{\sf PA}$ (see Section \ref{truththeo} for details) extended by further, possibly partial, predicate constants $P_1, P_2,\ldots$. The interpretation of the truth predicate will then consist of a set of G\"odel numbers of sentences of $\lt$, i.e., a set of integers.


To show that for every $\theta_\mathfrak{M}$-fixed point $f$ and all $J\in X$ the set $f(J)$ is $\Pi^1_1$-hard, we show that every $\Pi^1_1$-set $A$ is $1$-reducible to $f(J)$, i.e., there is a recursive $1-1$ function $\pi$ such that
\begin{align*}n\in A\Leftrightarrow \pi(n)\in f(J).\end{align*}
Indeed, following the outlines of \cite{wel15} \citep[see also][]{fhks15} we establish a stronger claim, that is, we show that $[\theta_\mathfrak{M}(f)](J)$ is $\Pi^1_1$-hard for any $\theta_{\mathfrak{M}}$-sound valuation function pair $(Y_f,f)$, where in this context $(Y_f,f)$ is called $\theta_{\mathfrak{M}}$-sound iff $Y_f$ and $f$ satisfy the conditions of Proposition \ref{GFP}, that is, iff starting with the set $Y_f$ and valuation function $f$ we can inductively define a fixed point of $\theta$.

\begin{lem}\label{LB}For ${\sf e}\in\{{\sf c},{\sf K3},{\sf N3},{\sf Nve},\}$ let $\mathfrak{M}_{\T}=(D,X\times Y_f,H_{\Phi_{\sf e}})$ be a truth structure with $Y_f\in {\sf Adm}_\mathfrak{M}$ and $(Y_f,f$ $\theta^{\Phi_{\sf e}}_{\mathfrak{M}}$-sound. Let $A$ be $\Pi^1_1$-set of integers. Then there exists a recursive $1-1$-function $\pi$ such that for all $J\in X$:
\begin{align*}n\in A\Leftrightarrow \pi(n)\in[\theta^{\Phi_{\sf e}}_\mathfrak{M}f](J).\end{align*}
\end{lem}

\begin{proof}See Appendix \ref{compl}.
\end{proof}

\begin{prop}\label{CplFP}Let $\mathfrak{M}_{\T}=(D,X\times Y_f,H_{\Phi_{\sf e}})$  be a Kripkean truth structure and $f$ a fixed point of $\theta_{\mathfrak{M}}$. Then $f(J)$ is a $\Pi^1_1$-hard set of integers for ${\sf e}\in\{{\sf c},{\sf K3},{\sf N3},{\sf Nve}\}$.
\end{prop}
\begin{proof}Direct corollary of Lemma \ref{LB}.\end{proof}

Proposition \ref{CplFP} provides us with a lower bound of the fixed points of the $\theta$-operator relative to a Kripkean Truth Structure. Notice, however, that the $\theta$-operator can be defined using a $\Pi^1_1$-formula. This means that the minimal fixed point defined on  the set $\mathcal{B}$ starting from the minimal valuation function $f_0$ with $f_0(J)=\emptyset$ for all $J\in X$ will be a $\Pi^1_1$-complete set, if the relevant parameters can be arithmetically defined.



As discussed in \cite{fhks15} Lemma \ref{LB}  implies that an $\mathbb{N}$-categorical axiomatization of the strong Kleene supervaluation fixed points is out of reach: it is not possible to single out these fixed point relative to an intended model using a r.e.~set of sentences. However, it is possible to, following the outlines of \cite{ste18a}, provide a variant construction that shares many characteristics of the present framework, but which is of reduced recursion-theoretic complexity and lends itself to $\mathbb{N}$-categorical axiomatizations in the sense of \cite{fhks15}. We leave a thorough discussion of the variant construction for another occasion.

\section{Modal strong Kleene supervaluation: indicative and subjunctive conditionals}\label{msks}
In the preceding sections we focussed on constructing a model for naive truth for a language comprising a conditional that can be employed for conditional reasoning in logic and mathematics, that is, the conditional is to be for the partial logic what material implication is for classical logic. However, indicative and subjunctive natural language conditionals are not, at least in most cases, governed by the rules and principles of material implication or the ${\sf N3}$-conditional. This shows in the paradoxes of material implications many of which apply to the ${\sf N3}$-conditional.\footnote{See, e.g., \cite{egc16,vfi11} for overview on the semantics of (natural language) conditionals.}

Assuming classical logic the most prominent semantic analysis of indicative and subjunctive conditionals conceives of them in terms of strict or variably-strict implication.\footnote{The notion of strict implication dates back to \cite{lew14}, whereas the notion of variably-strict implication originates with \cite{sta68} and \cite{lew73}.}  Both strict and variably-strict implication amount to forcing the corresponding material implication to be true in all words of a given set. The difference between strict and variably-strict implication is in how the set is specified and the set itself. In the same way we can define strict and variable-strict implication in a suitable possible world framework, one can define the strict and variably-strict ${\sf N3}$-conditional  on the basis of the ${\sf N3}$-conditional and thereby obtain a semantic interpretation for indicative and subjunctive conditionals. To provide an adequate semantics for natural language conditionals thus requires working in a modal semantics, more specifically, a version of possible world semantics enriched with an world-relative ordering relation on the set of possible worlds. In this section we show how to apply strong Kleene supervaluation to such semantics, that is, we show how to build supervaluation structures relative to so-called ordering frames and show how Kripkean truth structures can be constructed relative to these modal strong Kleene supervaluation structures.

From now on our base language will be the language $\lc$ which extends $\mathcal{L}$ by a conditional connective $\rhd$ and a one-place modal operators $\Box$. The set of well-formed formulas ${\sf Form}_{\lc}$ is specified by
\begin{align*}
\varphi&::=(s=t)\,|\,P^{m}_{j}t_1,\ldots,t_m\,|\,\Box\varphi\,|\,\neg\varphi\,|\,\varphi\wedge\varphi\,|\,\varphi\rightarrow\varphi)\,|\,\varphi\rhd\varphi\,|\,\forall x\varphi
\end{align*}
with $j,m\in\omega$, $s,t,t_1,\ldots,t_m\in{\sf Term}_\mathcal{L}$, and $x\in{\sf Var}_\mathcal{L}$. In addition to the notational abbreviations (defined symbols) we introduced for $\mathcal{L}$ 
we define the strict implication connective $\strictif$:
\begin{flalign*}
&(\strictif)&&\varphi\strictif\psi:\leftrightarrow\Box(\varphi\rightarrow\psi).&&
\end{flalign*}

A key notion in the semantics for natural-language conditionals such as indicative conditionals is that of an Ordering frame:\begin{de}[Ordering frame]Let $F=(W,R,\preceq,D)$ be a tuple with
 $W,D\neq\emptyset$, $R\subseteq W\times W$ and $\preceq\,\subseteq\leftidx{^{W}}{(W\times W)}$. Then $F$ is called an ordering frame iff for all $w\in W$:
 \begin{itemize}
 \item $\preceq_{w}$ is a weak total order on the set $\{v\in W\,|\,wRv\text{ or }v=w\}$.\footnote{Usually more assumption are made on behalf of $\preceq$, but these are irrelevant for our construction: further conditions can be freely imposed to the framework.}
 \end{itemize}
\end{de}

While a strong Kleene supervaluation structure for $\mathcal{L}$ was built on a non-empty domain, strong Kleene supervaluation structures for the modal language $\lc$, i.e., modal supervaluation structures will be constructed relative to an ordering frame. In this setting interpretation functions will assign interpretations to the descriptive vocabulary---henceforth denoted by ${\sf Descr}_\mathcal{L}$---relative to a possible world, i.e., relative to elements of $W$. 

\begin{de}[Interpretation, Ordering]Let $F$ be an ordering frame. A function $I$ with domain $W\times{\sf Descr}_\mathcal{L}$ for some language $\mathcal{L}$ is called a modal interpretation iff $I_{w}$ is an ${\sf k}$-interpretation for all $w\in W$ (cf.~Definition \ref{kinto}). The set of all interpretations over some ordering frame $F$ is denoted by ${\sf Int}_F$. For $I,J\in{\sf Int}_F$ we say $I\leq J$ iff $I^+(w,P)\subseteq J^+(w,P)$ and $I^-(w,P)\subseteq J^-(w,P)$ for all $w\in W$ and $P\in{\sf Pred}$.
\end{de}

\begin{de}[Modal supervaluation structure]A tuple $\mathfrak{M}=(F, X,H)$ is called a modal supervaluation structure iff
\begin{enumerate}[label=(\roman*)]
\item $F$ is an ordering frame;
\item $X\subseteq{\sf Int}_F$ such that $J(t,w)=I(t,v)$ for $J,I\in X$ and $w,v\in W$ and all closed terms $t$;
\item $H\subseteq X\times X$ such that
\begin{itemize}
\item $H$ is transitive and reflexive;
\item if $(x,y)\in H$, then $x\leq y$.
\end{itemize}
\end{enumerate}
\end{de}

(i) imposes that the interpretation of closed terms is rigid. This is to keep things as simple as possible. If this assumption is dropped one needs to deal with the usual complication that arise in quantified modal logic due to non-rigid terms.\footnote{See \cite{gar01} for discussion and \cite{hal21} for some further complications that arise for modal languages that contain a truth predicate.} Consequently, variable assignments are not world relative, that is, the notion of an assignment is that of Definition \ref{ass}.

 The denotation of a term is defined as in Definition \ref{Den} except for the additional (but innocuous) world parameter. With these definitions in place we can specify the truth conditions for formulas of the language $\lc$. We shall not repeat the truth conditions for the logical connective already spelled out in Definition \ref{TSV}. These remain unchanged but relativized to a world parameter, e.g., the truth condition of sentence of the form $\varphi\rightarrow\psi$ will be:
\begin{align*}\mathfrak{M},J,w\Vdash\varphi\rightarrow\psi[\beta]&\,{\rm iff}\,\forall J'((J,J')\in H\,\&\,\mathfrak{M},J',w\Vdash\varphi[\beta]\Rightarrow\mathfrak{M},J',w\Vdash\psi[\beta]).\end{align*}

\begin{de}[Truth in a modal supervaluation structure]\label{MTSV}
 Let $\mathfrak{M}=(F, X,H)$ be a modal supervaluation structure with $J\in X, w\in W$. Then a formula $\varphi\in\lc$ is true in $\mathfrak{M}$ at $w$ and $J$ under assignment $\beta$ ($\mathfrak{M},w,J\Vdash\varphi[\beta]$) iff
 \begin{align*}
 &\forall v\in W(wRv\Rightarrow \mathfrak{M},v,J\Vdash\psi[\beta]),&&\,{\rm if}\,\varphi\doteq\Box\psi;\\
 &\exists v\in W(wRv\,\&\,\mathfrak{M},v,J\Vdash\neg\psi[\beta]),&&\,{\rm if}\,\varphi\doteq\neg\Box\psi;\\
& {\sf Imposs}_{w}(\psi)\,{\rm or}\,\exists v(wRv\,\&\,\mathfrak{M},v,J\Vdash\psi[\beta]\,\&\,\forall u(u\preceq_{w} v\Rightarrow \mathfrak{M},u,J\Vdash\psi\rightarrow\chi[\beta]),&&{\rm if}\,\varphi\doteq\psi\rhd\chi;\\
&\exists v(wRv\,\&\,\mathfrak{M},v,J\Vdash\psi[\beta])\,\&\,&&\\
&\forall v(wRv\,\&\,\mathfrak{M},v,J\Vdash\psi[\beta]\Rightarrow\exists u(u\prec_{w}v\,\&\,wRv\,\&\,\mathfrak{M},u,J\Vdash\psi\wedge\neg\chi[\beta],&&{\rm if}\,\varphi\doteq\neg(\psi\rhd\chi).
\end{align*}
${\sf Imposs}_w(\psi)$ is short for $\forall v(wRv\Rightarrow\mathfrak{M},v,J\Vdash\neg\psi)$. Finally, we write $\mathfrak{M},w,J\Vdash\varphi$ iff $\mathfrak{M},w,J\Vdash\varphi[\beta]$ holds for all assignments $\beta$ and $\mathfrak{M},w\Vdash\varphi$ ($\mathfrak{M}\Vdash\varphi$) iff $\mathfrak{M},w,J\Vdash\varphi$ for all $J\in X$ (and $w\in W$).\footnote{To be precise, for Definition \ref{MTSV} to be well-formed the truth-conditions for $\psi\rhd\chi$ and $\neg(\psi\rhd\chi)$, should say
$$\,\forall J'((J,J')\in H\Rightarrow(\mathfrak{M},u,J'\Vdash\psi[\beta]\Rightarrow\mathfrak{M},u,J'\Vdash\chi[\beta])$$
and
$$\mathfrak{M},u,J\Vdash\psi[\beta]\text{ and }\mathfrak{M},u,J\Vdash\psi\neg\chi[\beta]$$
where we have written $\mathfrak{M},u,J\Vdash\psi\rightarrow\chi[\beta]$ and $\mathfrak{M},u,J\Vdash\psi\wedge\neg\chi[\beta]$ respectively.}
\end{de}
Definition \ref{MTSV} gives truth-conditions for the $\rhd$ that are widely endorsed in the literature \citep[cf., e.g.,][for discussion]{egc16,vfi11}. We take it that to the extent that $\rightarrow$ is an adequate conditional connective for mathematical reasoning, $\rhd$ (or, possibly, $\strictif$) are adequately representing indicative and subjunctive natural language conditionals.

\subsection{Modal Truth Structures}
With the semantics for $\lc$ in place modal truth structures for the language $\lct$ can be constructed along the lines of Section \ref{TS} with the exception that all definitions need to be generalized to allow for a world parameter. As a consequence the construction follows the pattern of the generalization of Kripke's theory of truth to modal languages proposed in \cite{haw09,ste14b,ste15a,ste24ff}. We shall state the most important definitions and propositions, but will omit proofs, as these are immediate generalizations of those of Sections \ref{TS} and \ref{KTS}.

In the modal setting valuation functions are no longer functions from ${\sf Int}_F$ to $\mathcal{P}({\sf Sent})$ but functions from ${\sf Int}_F\times W$ to $\mathcal{P}({\sf Sent})$:
\begin{de}[Admissible Valuations---modal case]\label{madval} A valuation on a supervaluation structure $\mathfrak{M}=(F,X,H)$ is a function $f:W\times X\rightarrow\mathcal{P}({\sf Sent}_{\lt})$. The set of all valuation on $\mathfrak{M}$ is denoted by ${\sf Val}_\mathfrak{M}$. A valuation function is admissible for $\mathfrak{M}$ iff
\begin{enumerate}[label=(\roman*)]
\item $f$ is consistent, that is, $f(w,J)\cap\overline{f(w,J)}=\emptyset$ for all $w\in W$ and $J\in X$ where
$$\overline{f(w,J)}:=\{\varphi\,|\,\neg\varphi\in f(w,J)\};$$
\item for all $J\in X, w\in W$ and $\varphi\in\mathcal{L}$:
\begin{align*}\text{if }\varphi\in f(w,J),&\text{ then }\mathfrak{M},w,J\Vdash\varphi\end{align*}
\item for all $J,J'\in X$, if $(J,J')\in H$, then $f(w,J)\subseteq f(w,J')$ for all $w\in W$.
\end{enumerate}
The set of admissible valuation functions for a given model $\mathfrak{M}$ is denoted by $\mathcal{B}_\mathfrak{M}$.\end{de}

Of course, the addition of the world parameter also means that the definition of the ordering on ${\sf Val}_\mathfrak{M}$ has to be modified: for $f,g\in{\sf Val}_\mathfrak{M}$ set
\begin{align*}f\leq g&:\leftrightarrow\forall J\in X\forall w\in W(f(w,J)\subseteq g(w,J)).\end{align*}
In contrast, the definition of the admissibility condition $\Phi$ and the ordering $\leq_\Phi$ remain unaltered and are simply those of Definition \ref{ordac}. We can then define a (grounded) modal truth structure as in Definition \ref{dets} but where the truth structure is built on a modal supervaluation structure, i.e., a tuple consisting of a frame, a set of interpretations on the frame and an admissibility relation. 

\begin{de}[Truth Structure, Grounded Truth Structure---modal case]Let $\mathfrak{M}=(F,X,H)$ be a supervaluation structure. Then the tupel $(F,X\times Y,H_\Phi)$ with $Y\subseteq\mathcal{B}_\mathfrak{M}$ and where $H_\Phi$ is restricted to $X\times Y$ is called a truth structure over $\mathfrak{M}$. If $Y\subseteq{\sf Val}_{\mathfrak{M}}$ has a $\leq$-minimal valuation function, that is, an $f$ in $Y$ such that $f\leq g$ for all $g\in Y$ and $Y\cap\Phi(f)\neq\emptyset$, then $(F,X\times Y,H_\Phi)$ is called a grounded truth structure and $Y$ is called a grounded set of valuations. The set of all grounded sets of valuations is denoted by ${\sf Adm}_\mathfrak{M}$.
\end{de}

We can then use the construction of Section \ref{KTS} to generalize Propositions \ref{GFP}, \ref{CFPT} and Corollaries \ref{CFPTM}, \ref{KSS} to show the existence of Kripkean truth structure for the language $\lct$.  To this effect we let $\mathcal{B}_{\T}$ be the set of valuation functions $f$ such that for all $J\in X$:
\begin{itemize}
\item there a exists a valuation function $f_\Box:W\rightarrow\mathcal{P}({\sf Sent}_{\mathcal{L}_{\Box}}$ such that $f_\Box$ s a modal strong Kleene fixed point on the frame $F$ in the sense of \cite{haw09,ste14b,ste15a} and $f(w,J)\cap\mathcal{L}_\Box=f_{\Box}(w)$ and where $\mathcal{L}_{\Box}$ is the $\rightarrow$- and $\rhd$-free fragment of $\mathcal{L}^{\rhd}_{\T}$.
\item $f(w,J)$ is ${\sf K3}$-saturated for all $w\in W$.
\end{itemize}
 We obtain the following results:

\begin{prop}[Modal Fixed Points]\label{MCFP}Let $\mathfrak{M}=(F,X,H)$ be a modal supervaluation structure and ${\sf e}\in\{{\sf c},{\sf K3},{\sf N3},{\sf Nve}\}$. Let $f\in \mathcal{B}^{\T}_\mathfrak{M}$ be minimal in $\mathcal{B}^{\T}_\mathfrak{M}$ and set $Y_f=\mathcal{B}^{\T}_\mathfrak{M}$. Then there exists $g\in Z_g\subseteq Y_f$ such that $\theta^{\Phi_{\sf e}}_\mathfrak{M}(Z_g,g)=g$ and $\Theta^{\Phi_{\sf e}}_\mathfrak{M}(Z_g)=Z_g$.
\end{prop}

\begin{prop}[Kripkean Truth Structures for $\lct$]\label{MCFPTM}Let $\mathfrak{M}=(F,X,H)$ be a modal supervaluation structure and ${\sf e}\in\{{\sf c},{\sf K3},{\sf N3},{\sf Nve}\}$.  Then there exists a grounded truth set $Y_f$ and admissible valuation function $f$ such that for all $\varphi\in{\sf Sent}_{\lt^{\rhd}}$
\begin{align*}(F,X\times Y_f,H_{\Phi_{\sf e}}),J_f,w\Vdash\varphi&\text{ iff }(D,X\times Y_f,H_{\Phi_{\sf e}}),J_f,w\Vdash\T\gn\varphi\end{align*}
for all $J\in X$ and $w\in W$.\end{prop}

As in the non-modal case we can find a reflexive Kripkean truth structure by constructing a $\Phi_{\sf e}$-generated substructure:

\begin{cor}Let ${\sf e}\in\{{\sf K3},{\sf N3},{\sf Nve}\}$ and $(F,X\times Y_f,H_{\Phi_{\sf e}})$ be a modal Kripkean Truth structure and let $Y^{\leq_{\Phi_{\sf e}}}_f:=\{g\in Y_f\,|\,f\leq_{\Phi_{\sf e}} g\}$. Then the structure $(F,X\times Y^{\leq_{\Phi_{\sf e}}}_f,H_{\Phi_{\sf e}})$ is a modal Kripkean truth structure such that $H_{\Phi_{\sf e}}$ is reflexive on $X\times Y^{\leq_{\Phi_{\sf e}}}_f$.\end{cor}

The propositions establish that the strong Kleene supervaluation framework can be used to construct Kripkean truth models for natural language conditionals as well as ``logical'' conditionals, i.e., conditionals that are use to formalize conditional reasoning in logic and mathematics.This suggests that the framework is well suited for interpreting interesting and expressively rich fragments of natural language, i.e., fragments that allow for self-reference but also enable us to express different conditionals and conditional reasoning.

\section{Conclusion and Outlook}
In this paper we set out to apply a Kripke's theory of truth to languages that comprise adequate logical conditionals as well as indicative and subjunctive conditionals. This led us to the framework of strong Kleene supervaluation and modal strong Kleene supervaluation for which we showed the existence of Kripkean truth models. Importantly, the framework is much wider: strong Kleene supervaluation can be applied not only to conditionals but to other non-monotone notions. Indeed, as we discuss in companion paper the framework lends itself particular well to giving a semantics and truth theory for $\langle 1,1\rangle$-quantifiers.

There remain open questions and many possible further research avenues to explore.  For one, as mentioned in Section \ref{truththeo} in future work we intend to provide fixed point results for non-compact admissibility relations according to which admissible precisifications are required to be $\omega$-complete saturated sets that have the Henkin property. This would guarantee that both the universal and the existential quantifier commute with the truth predicate and open up the possibility of attractive axiomatic truth theories formulated in ${\sf N3}$ Nelson logic. We will discuss such theories in a sequel to this study. For another, it seems interesting to investigate variants of our framework that use alternative falsity conditions for the conditional, in particular, the falsity condition employed in intuitionistic logic. Arguably, this would re-establish symmetry between the ruth and falsity condition of the conditional. Moreover, if symmetric strong Kleene logic ${\sf KS3}$ is assumed, the resulting conditional should allow for contraposition. However, the details require carefully study and are left for future research.

\bibliography{/users/js16807/Dropbox/Uni/Bibliography/LITA-3}

\begin{thebibliography}{}

\bibitem[Bacon, 2013]{bac13}
Bacon, A. (2013).
\newblock A new conditional for naive truth theory.
\newblock {\em Notre Dame Journal of Formal Logic}, 54(1):87--104.

\bibitem[Beall, 2009]{bea09}
Beall, J. (2009).
\newblock {\em Spandrels of truth}.
\newblock Oxford University Press.

\bibitem[Beall, 2021]{bea21}
Beall, J. (2021).
\newblock {Transparent Truth as a Logical Property}.
\newblock In Lynch, M.~P., Wyatt, J., Kim, J., and Kellen, N., editors, {\em
  {The Nature of Truth: Classic and Contemporary Perspectives}}. The MIT Press.

\bibitem[Beall et~al., 2018]{bgr18}
Beall, J., Glanzberg, M., and Ripley, D. (2018).
\newblock {\em Formal theories of truth}.
\newblock Oxford University Press.

\bibitem[Campbell-Moore, 2021]{cam21}
Campbell-Moore, C. (2021).
\newblock Indeterminate truth and credences.
\newblock In Nicolai, C. and Stern, J., editors, {\em Modes of Truth}, pages
  182--208. Routledge.

\bibitem[Cresswell, 1975]{cre75}
Cresswell, M.~J. (1975).
\newblock Hyperintensional logic.
\newblock {\em Studia Logica}, 34(1):25--38.

\bibitem[Egr\'e and Cozic, 2016]{egc16}
Egr\'e, P. and Cozic, M. (2016).
\newblock Conditionals.
\newblock In Aloni, M. and Dekker, P., editors, {\em Handbook of Formal
  Semantics}, pages 490--524. Cambridge University Press.

\bibitem[Feferman, 1960]{fef60}
Feferman, S. (1960).
\newblock Arithmetization of metamathematics in a general setting.
\newblock {\em Fundamenta Mathematicae}, XLIX:35--92.

\bibitem[Feferman, 1984]{fef84}
Feferman, S. (1984).
\newblock Toward useful type-free theories {I}.
\newblock {\em The Journal of Symbolic Logic}, 49(1):75--111.

\bibitem[Field, 2008]{fie08}
Field, H. (2008).
\newblock {\em Saving Truth from Paradox}.
\newblock Oxford University Press.

\bibitem[Field, 2016]{fie16}
Field, H. (2016).
\newblock Indicative conditionals, restricted quantification, and naive truth.
\newblock {\em The Review of Symbolic Logic}, 9(1):181--208.

\bibitem[Field, 2021]{fie21}
Field, H. (2021).
\newblock Properties, propositions, and conditionals.
\newblock {\em Australasian Philosophical Review}, forthcoming.

\bibitem[Fischer et~al., 2015]{fhks15}
Fischer, M., Halbach, V., Kriener, J., and Stern, J. (2015).
\newblock Axiomatizing semantic theories of truth?
\newblock {\em The Review of Symbolic Logic}, 8(2):257--278.

\bibitem[Garson, 2001]{gar01}
Garson, J.~W. (2001).
\newblock Quantification in {M}odal {L}ogic.
\newblock In Gabbay, D. and Guenther, F., editors, {\em Handbook of
  Philosophical Logic}, volume~3, pages 267--323. Kluwer Academic Publishers.
\newblock 2nd Edition. First published in 1984.

\bibitem[Glanzberg, 2004]{gla04}
Glanzberg, M. (2004).
\newblock A contextual-hierarchical approach to truth and the liar paradox.
\newblock {\em Journal of Philosophical Logic}, 33:27--88.

\bibitem[Gupta, 1982]{gup82}
Gupta, A. (1982).
\newblock Truth and paradox.
\newblock {\em Journal of Philosophical Logic}, 11:1--60.

\bibitem[Gupta and Belnap, 1993]{beg93}
Gupta, A. and Belnap, N. (1993).
\newblock {\em The revision theory of truth}.
\newblock The MIT Press.

\bibitem[Halbach, 2021]{hal21}
Halbach, V. (2021).
\newblock The fourth grade of modal involvement.
\newblock In Nicolai, C. and Stern, J., editors, {\em Modes of Truth: The
  Unified Approach to Modality, Truth, and Paradox}, pages 209--230. Routledge.

\bibitem[Halbach and Welch, 2009]{haw09}
Halbach, V. and Welch, P. (2009).
\newblock Necessities and necessary truths: A prolegomenon to the use of modal
  logic in the analysis of intensional notions.
\newblock {\em Mind}, 118:71--100.

\bibitem[Herzberger, 1982]{her82a}
Herzberger, H. (1982).
\newblock Naive semantics and the liar paradox.
\newblock {\em The Journal of Philosophy}, 79:479--497.

\bibitem[Iacona and Rossi, 2024]{iaro24}
Iacona, A. and Rossi, L. (2024).
\newblock Na{\"\i}ve truth and the evidential conditional.
\newblock {\em Journal of Philosophical Logic}, 53(2):559--584.

\bibitem[Kremer, 1988]{kre88}
Kremer, M. (1988).
\newblock Kripke and the logic of truth.
\newblock {\em Journal of Philosophical Logic}, 17(3):225--278.

\bibitem[Kripke, 1975]{kri75}
Kripke, S. (1975).
\newblock Outline of a theory of truth.
\newblock {\em The Journal of Philosophy}, 72:690--716.

\bibitem[Leitgeb, 2019]{lei19}
Leitgeb, H. (2019).
\newblock Hype: A system of hyperintensional logic (with an application to
  semantic paradoxes).
\newblock {\em Journal of Philosophical Logic}, 48(2):305--405.

\bibitem[Lewis, 1914]{lew14}
Lewis, C. (1914).
\newblock The calculus of strict implication.
\newblock {\em Mind}, 23:240--247.

\bibitem[Lewis, 1973]{lew73}
Lewis, D. (1973).
\newblock {\em Counterfactuals}.
\newblock Wiley-Blackwell.

\bibitem[McGee, 1991]{mcg91}
McGee, V. (1991).
\newblock {\em Truth, {V}agueness and {P}aradox}.
\newblock Hackett Publishing Company, Indianapolis.

\bibitem[Nicolai and Stern, 2021]{nist21}
Nicolai, C. and Stern, J. (2021).
\newblock The {M}odal {L}ogic of {K}ripke-{F}eferman {T}ruth.
\newblock {\em The Journal of Symbolic Logic}, 86(1):362--395.

\bibitem[Odintsov and Wansing, 2021]{odwa21}
Odintsov, S. and Wansing, H. (2021).
\newblock Routley star and hyperintensionality.
\newblock {\em Journal of Philosophical Logic}, 50(1):33--56.

\bibitem[Paoli, 2013]{pao13}
Paoli, F. (2013).
\newblock {\em Substructural logics: a primer}.
\newblock Springer Science \& Business Media.

\bibitem[Restall, 2002]{res02}
Restall, G. (2002).
\newblock {\em An introduction to substructural logics}.
\newblock Routledge.

\bibitem[Rogers, 1967]{rog67}
Rogers, H. (1967).
\newblock {\em Theory of recursive functions and effective computability}.
\newblock McGraw-Hill, New York.

\bibitem[Rossi, 2016]{ros16}
Rossi, L. (2016).
\newblock Adding a conditional to kripke?s theory of truth.
\newblock {\em Journal of Philosophical Logic}, 45(5):485--529.

\bibitem[Schlenker, 2010]{schl10}
Schlenker, P. (2010).
\newblock Super liars.
\newblock {\em The Review of Symbolic Logic}, 3(3):374--414.

\bibitem[Simmons, 2018]{sim18}
Simmons, K. (2018).
\newblock {\em Semantic singularities: paradoxes of reference, predication, and
  truth}.
\newblock Oxford University Press.

\bibitem[Stalnaker, 1968]{sta68}
Stalnaker, R. (1968).
\newblock A theory of conditionals.
\newblock {\em Studies in logical theory}, 2:98--112.

\bibitem[Stern, 2014]{ste14b}
Stern, J. (2014).
\newblock Modality and {A}xiomatic {T}heories of {T}ruth {I}{I}:
  {K}ripke-{F}eferman.
\newblock {\em The Review of Symbolic Logic}, 7(2):299--318.

\bibitem[Stern, 2016]{ste15a}
Stern, J. (2016).
\newblock {\em Toward Predicate Approaches to Modality}, volume~44 of {\em
  Trends in Logic}.
\newblock Springer, Switzerland.

\bibitem[Stern, 2018]{ste18a}
Stern, J. (2018).
\newblock Supervalutation-{S}tyle {T}ruth {W}ithout {S}upervaluations.
\newblock {\em Journal of Philosophical Logic}, 47(5):817--850.

\bibitem[Stern, 2021]{ste21}
Stern, J. (2021).
\newblock Belief, truth, and ways of believing.
\newblock In Nicolai, C. and Stern, J., editors, {\em Modes of Truth: The
  Unified Approach to Modality, Truth, and Paradox}, pages 151--181. Routledge.

\bibitem[Stern, 2024]{ste24ff}
Stern, J. (2024).
\newblock The liar paradox and modalities.
\newblock In Rossi, L., editor, {\em The Liar Paradox}. CUP.
\newblock forthcoming.

\bibitem[Tarski, 1935]{tar35}
Tarski, A. (1935).
\newblock Der {W}ahrheitsbegriff in den formalisierten {S}prachen.
\newblock In Berka, K. and Kreiser, L., editors, {\em Logik-Texte}, pages
  445--546. Berlin.
\newblock 1971.

\bibitem[Thomason, 1969]{tho69}
Thomason, R.~H. (1969).
\newblock A semantical study of constructible falsity.
\newblock {\em Mathematical Logic Quarterly}, 15(16-18):247--257.

\bibitem[Visser, 1989]{vis89}
Visser, A. (1989).
\newblock Semantics and the {L}iar {P}aradox.
\newblock In Gabbay, D., editor, {\em Handbook of Philosophical Logic}, pages
  617--706. Dordrecht.

\bibitem[von Fintel, 2011]{vfi11}
von Fintel, K. (2011).
\newblock Conditionals.
\newblock In von Heusinger, K., Maienborn, C., and Portner, P., editors, {\em
  Semantics: An International Handbook of Meaning}, volume~2, pages 1515--1538.
  DeGruyter.

\bibitem[Wansing, 1991]{wan91}
Wansing, H. (1991).
\newblock {\em The Logic of Information Structures}.
\newblock Springer, Berlin.

\bibitem[Wansing, 2001]{wan01}
Wansing, H. (2001).
\newblock Negation.
\newblock In Goble, L., editor, {\em The Blackwell guide to philosophical
  logic}, pages 415--436. Wiley Online Library.

\bibitem[Welch, 2015]{wel15}
Welch, P. (2015).
\newblock The {C}omplexity of the {D}ependence {O}perator.
\newblock {\em Journal of Philosophical Logic}, 44(3):337--340.

\bibitem[Yablo, 2013]{yab03}
Yablo, S. (2013).
\newblock New {G}rounds for {N}aive {T}ruth {T}heories.
\newblock In Beall, J., editor, {\em Liars and Heaps}, pages 312--330. OUP.

\end{thebibliography}

\appendix

\section{Sequent Calculi}\label{SeqC}
We introduce the sequent calculi of the various logics we appealed to throughout the paper, and state some important properties of these calculi. We start with strong Kleene logic ${[\sf K3}$.
\begin{de}[${\sf K3}$-Sequent calculus]Let $\mathcal{L}$ be an arbitrary first-order formula whose logical constants include $\neg,\wedge,\forall$. Let $\Gamma,\Delta\subset{\sf Frml}_\mathcal{L}$ be finite sets of formulae and $\varphi,\psi\in{\sf Frml}_\mathcal{L}$. The ${\sf K3}$-calculus is given by the following rules:
\begin{align*}
	& \text{{\small{\sc (ax)}}}\;\;\;\AxiomC{$\Gamma, \varphi \Rightarrow\varphi,\Delta$}\noLine\UnaryInfC{\textup{for $\varphi$ a literal}}  \DisplayProof
		&&\AxiomC{$\Gamma\Rightarrow\Delta,\varphi$}\AxiomC{$\varphi,\Gamma\Rightarrow\Delta$}\Llb{cut}
		\BinaryInfC{$\Gamma \Rightarrow\Delta$}\DisplayProof\\[10pt]
	&\AxiomC{$\Gamma\Rightarrow\chi,\Delta$}\Llb{$\neg$l}\UnaryInfC{$\Gamma,\neg\chi\Rightarrow\Delta$}\DisplayProof&&\\[10pt]
	& \AxiomC{$\Gamma,\varphi\Rightarrow\Delta$}\Llb{dn-l}
	\UnaryInfC{$\Gamma,\neg \neg \varphi \Rightarrow\Delta$}
	\DisplayProof
		&& \AxiomC{$\Gamma\Rightarrow\varphi,\Delta$}\Llb{dn-r}
	\UnaryInfC{$\Gamma\Rightarrow\neg \neg \varphi,\Delta$}
	\DisplayProof\\[10pt]
	& \AxiomC{$\Gamma,\neg\varphi\Rightarrow\Delta$}
		\AxiomC{$\Gamma\Rightarrow\neg\psi,\Delta$}\Llb{$\neg\land$l}
		\BinaryInfC{$\Gamma,\neg (\varphi\land \psi) \Rightarrow\Delta$}\DisplayProof
		&&\AxiomC{$\Gamma\Rightarrow\neg\varphi_i, \Delta$}\Llb{$\neg\land${r}}\RightLabel{$i=0,1$}
	\UnaryInfC{$\Gamma\Rightarrow\Delta,\neg (\varphi_0\land \varphi_1)$}\DisplayProof\\[10pt]
	&\AxiomC{$\Gamma,\varphi_i\Rightarrow\Delta$}\Llb{$\land$l}\UnaryInfC{$\Gamma,\varphi_0\land \varphi_1\Rightarrow\Delta$}\DisplayProof
		&&\AxiomC{$\Gamma\Rightarrow\varphi,\Delta$}\AxiomC{$\Gamma\Rightarrow\psi,\Delta$}\Llb{$\land$r}\BinaryInfC{$\Gamma \Rightarrow\Delta,\varphi\land \psi$}\DisplayProof\\[10pt]
			&\AxiomC{$\Gamma,\neg\varphi(y)\Rightarrow\Delta$}\Llb{$\neg\forall$l}\UnaryInfC{$\Gamma,\neg\forall v\varphi(v)\Rightarrow\Delta$}\DisplayProof&&\AxiomC{$\Gamma\Rightarrow\neg\varphi(t),\neg\forall v\varphi(v),\Delta$}\Llb{$\neg\forall$r}\UnaryInfC{$\Gamma\Rightarrow\neg\forall v\varphi(v),\Delta$}\DisplayProof\\[10pt]
		&\AxiomC{$\Gamma,\forall v\varphi(v),\varphi(t)\Rightarrow\Delta$}\Llb{$\forall$l}\UnaryInfC{$\Gamma,\forall v\varphi(v)\Rightarrow\Delta$}\DisplayProof&&\AxiomC{$\Gamma\Rightarrow\varphi(y),\Delta$}\Llb{$\forall$r}\UnaryInfC{$\Gamma\Rightarrow\forall v\varphi(v),\Delta$}\DisplayProof
		\end{align*}
The formula $\chi$, in the rule $(\neg\text{{\sc l}})$, is required to be a literal, i.e., an atomic or negated atomic formula of the language. In the rules for the quantifiers $y$ is required to be an eigenvariable and $v$ needs to be free for $y$ ($t$ respectively) in $\varphi$. 
\end{de}

We note that ({\sc ax}) and $(\neg\text{{\sc l}})$ can be lifted to all formulas of the language. Constant domain Nelson logic ${\sf N3}$ extends ${\sf K3}$ by rules for a conditional connective $\rightarrow$. A rigorous sequent calculus, and completeness proof was given by \cite{tho69} relative to a constant domain Kripke semantics. Our calculus diverges from Thomason's formulation and generalizes a single conclusion formulation for the propositional fragment by \cite{wan91}. As we argue below our calculus is equivalent to that of Thomason.

\begin{de}[{\sf N3}--Nelson Logic]${\sf N3}$ extends the sequent calculus of ${\sf K3}$ by the following rules for the conditional:
\begin{align*}
&\AxiomC{$\Gamma,\chi\Rightarrow\Delta$}\Llb{$\neg\!\!\rightarrow$r}\UnaryInfC{$\Gamma,\neg(\varphi\rightarrow\psi)\Rightarrow\Delta$}\DisplayProof&&\AxiomC{$\Gamma\Rightarrow \varphi,\Delta$}\AxiomC{$\Gamma\Rightarrow\neg \psi,\Delta$}\Llb{$\neg\!\!\rightarrow$r}\BinaryInfC{$\Gamma\Rightarrow\neg(\varphi\rightarrow\psi),\Delta$}\DisplayProof\\[10pt]
&\AxiomC{$\Gamma\Rightarrow \varphi,\Delta$}\AxiomC{$\Gamma,\psi\Rightarrow\Delta$}\Llb{$\rightarrow$l}\BinaryInfC{$\Gamma,\varphi\rightarrow\psi\Rightarrow\Delta$}\DisplayProof
&&\AxiomC{$\Gamma,\varphi\Rightarrow\psi$}\Llb{$\rightarrow$r}\UnaryInfC{$\Gamma\Rightarrow \varphi\rightarrow\psi,\Delta$}\DisplayProof
\end{align*}
with $\chi\in\{\varphi,\neg\psi\}$.
\end{de}

All initial sequents and rules are valid in Thomason's \citeyearpar{tho69} Kripke semantics for constant domain ${\sf N3}$. Moreover, all structural rules of Thomason's calculus are derivable in our calculus and one can also show that Thomason's elimination rules are admissible in our calculus. Thomason's completeness result then shows that the two calculi are equivalent.

\begin{de}[${\sf K3}\T$--${\sf K3}$ truth logic ]${\sf K3}\T$ extends the sequent calculus of ${\sf K3}$ by the following rules for the truth predicate:
\begin{align*}&\AxiomC{$\Gamma,\neg\varphi\Rightarrow\Delta$}\Llb{$\neg\T$l}\UnaryInfC{$\Gamma,\neg\T\gn\varphi\Rightarrow\Delta$}\DisplayProof&&\AxiomC{$\Gamma\Rightarrow\neg\varphi,\Delta$}\Llb{$\neg\T$r}\UnaryInfC{$\Gamma\Rightarrow\neg\T\gn\varphi,\Delta$}\DisplayProof\\[10pt]
&\AxiomC{$\Gamma,\varphi\Rightarrow\Delta$}\Llb{$\T$l}\UnaryInfC{$\Gamma,\T\gn\varphi\Rightarrow\Delta$}\DisplayProof&&\AxiomC{$\Gamma\Rightarrow\varphi,\Delta$}\Llb{$\T$r}\UnaryInfC{$\Gamma\Rightarrow\T\gn\varphi,\Delta$}\DisplayProof
\end{align*}
\end{de}
${\sf K3}\T$ is the ``logic'' corresponding to the $\Phi_{\sf Nve}$-admissibility condition and ultimately a variation on Kremer's \citeyearpar{kre88} logic of truth.

Finally, we outline how to extend all these calculi by a classical identity predicate.
\begin{de}[Identity]The rules for identity in ${\sf K3},{\sf N3}$, and ${\sf K3}\T$ are as follows for $\mathcal{L}$-terms $s,t$:
\begin{align*}
&\AxiomC{$\Gamma,t=t\Rightarrow\Delta$}\Llb{${\rm Ref}$}\UnaryInfC{$\Gamma\Rightarrow\Delta$}\DisplayProof&&\\[10pt]
&&&\AxiomC{$\Gamma,s=t\Rightarrow\Delta$}\Llb{$=\!\neg$r}\UnaryInfC{$\Gamma\Rightarrow s\neq t,\Delta$}\DisplayProof\\[10pt]
&\AxiomC{$\Gamma,\varphi(t)\Rightarrow\Delta$}\Llb{${\rm Rep}$}\UnaryInfC{$\Gamma,s=t,\varphi(s)\Rightarrow\Delta$}\DisplayProof&&
\end{align*}
\end{de}
As suggested, over ${\sf K3}$ and its extension the rules for identity have the effect of making identity classical, i.e., one can prove
\begin{align*}s=t,s\neq t&\Rightarrow\\
&\Rightarrow s=t,s\neq t\end{align*}
for all terms $s$ and $t$. 

\section{Proof of Lemma \ref{LB}}\label{compl}
As mentioned the general outlines of the proof are due to \cite{wel15} and, largely, follows the presentation of Welch's result in \cite{fhks15}. However, some minor tweaks need to be made to carry out the argument in the strong Kleene supervaluation setting. Let ${\sf Seq}$ be the set of codes of finite sequences. Then we know that there exists a recursive relation $R(u,n)\subseteq {\sf Seq}\times\mathbb{N}$ such that
\begin{flalign*}&(\dagger)&&n\in A\leftrightarrow\forall l\in \leftidx{^\omega}{\omega}\,\exists k_0\,\forall k\geq k_0\,(\neg R((l\restr k),n)),&&\end{flalign*}where $(l\restr k)$ is short for the code of the finite sequence  $( l(0),...,l(k) )$.\footnote{This follows from the following normal-form theorem by Kleene \citep[cf.][\S 16.1, Corollary V]{rog67}. Let $A$ be $\Pi^1_1$-set and $A$ a recursive relation. Then
\begin{align*}n\in A&\Leftrightarrow\forall l\in \leftidx{^\omega}{\omega}\,\exists k(R((l\restr k),n)).\end{align*}
See \cite{wel15} for further explanation.}
Since for $J\in X$ the set $[\theta_\mathfrak{M}(f)](J)$ is supposed to consist of G\"odel numbers of sentences we now introduce an alternative coding of finite sequences. A sequence $u=(u_0,\ldots,u_n)$ will now be coded by 
\begin{align*}\#((\underbrace{\kappa\wedge\ldots\wedge\kappa}_{u_0+1\text{-times}})\vee\ldots\vee(\underbrace{\kappa\wedge\ldots\wedge\kappa}_{u_n+1\text{-times}}))\end{align*}
where $\kappa$ denotes the standard Curry sentence introduced at the beginning of the paper.\footnote{\cite{wel15} does not work with the Curry sentence but the Liar sentence. The switch is required since, if we are working with $\Phi_{\sf Nve}, \Phi_{\omega{\sf Nve}}$, or, more generally, start the construction with the set $Y_{\T}$, no admissible precisification of a  $\theta_\mathfrak{M}$-sound valuation function $f$ will contain the Liar sentence (or its negation).}

 We denote this deviant coding by $^*$. The resulting set of sequence numbers ${\sf Seq}^*$ remains recursive and we can replace $R$ by a recursive relation $R^*$ in $(\dagger)$ such that
\begin{flalign*}&(\ddagger)&&n\in A\Leftrightarrow\forall l\in \leftidx{^{\omega}}{\omega}\,\exists k_0\,\forall k\geq k_0\,(\neg R^*((l\restr k)^*,n))&&\end{flalign*}
Next let $\sigma_n$ be the $\lt$-sentence

\begin{align*}[\exists u\,\T^*(u)\wedge\forall u,v\,(\T^*(u)\wedge\T^*(v)\rightarrow((u\subseteq v\vee v\subseteq u)\wedge\exists u'\,(\T^*(u')\wedge u\subset u')]\rightarrow\\
\exists u\,(\T^*(u)\wedge\neg R^*(u,\overline{n}))\end{align*} 
where $\T^*(u)$ is short for $\T u\wedge {\sf Seq}^*(u)$, and where ${\sf Seq}^*$ and $R^*$ need to be understood in terms of their representing formulae.
We now show for $J\in X$ that 	
\begin{flalign*}&({\rm Claim})&&n\in A\Leftrightarrow\#\sigma_n\in[\theta_\mathfrak{M}(f)](J)&&\end{flalign*}
For the left-to-right direction we assume $n\in A$. We need to show that $\mathfrak{M}_{\T},J^f\Vdash\sigma_n$, that is, we need to show that 
\begin{align*}\mathfrak{M}_{\T},J'^g\Vdash\exists u\,\T^*(u)\wedge\forall u,v\,(\T^*(u)\wedge\T^*(v)\rightarrow((u\subseteq v\vee v\subseteq u)\wedge\exists u'\,(\T^*(u')\wedge u\subset u')\end{align*}
implies
\begin{align*}\mathfrak{M}_{\T},J'^g\Vdash\exists u\,(\T^*(u)\wedge\neg R^*(u,\overline{n}))\end{align*}
for all truth interpretations $J'^g$ such that $(J^f,J'^g)\in H_{\Phi_{\sf e}}$. If $g(J')$ does not contain the codes of the finite initial segments of some infinite sequence then there is nothing to show because the antecedent of $\sigma_n$ will not be true. Assume otherwise, i.e., $(l\restr k)^*\in g(J')$ for some function $l$, for infinitely many $k$. By $n\in A$ and $(\ddagger)$ it follows that there must be an $k\in\omega$ such that $\neg R^*((l\restr k)^*, n)$. This establishes the left-to-right direction for ${\sf e}\in\{{\sf c},{\sf K3},{\sf N3},{\sf Nve},\omega{\sf N3},\omega{\sf Nve}\}$.

For the converse direction assume $n\not\in A$. We need to show that $\#\sigma_n\not\in[\theta_\mathfrak{M}(f)](J)$, that is we need to find a valuation function $g$ such that $(J^f,J'^g)\in H_{\Phi_{\sf e}}$ for some $J'\in X$ with $(J,J')\in H$: 
\begin{align*}\mathfrak{M}_{\T},J'^g\Vdash[\exists u\,\T^*(u)\wedge\forall u,v\,(\T^*(u)\wedge\T^*(v)\rightarrow((u\subseteq v\vee v\subseteq u)\wedge\exists u'\,(\T^*(u')\wedge u\subset u')].\end{align*}
but
\begin{align*}\mathfrak{M}_{\T},J'^g\not\Vdash\exists u\,(\T^*(u)\wedge\neg R^*(u,\overline{n})).\end{align*}
From $n\not\in A$ we know that there exists a function $l$ such that $\forall k R^*((l\restr k)^*,n)$. By assumption and Proposition \ref{GFP} we know that there is a $\theta^{\Phi_{\sf e}}_{\mathfrak{M}}$-fixed point $g_0$ such that $(J^f,J'^{g_0})\in H_{\Phi_{\sf e}}$. For ${\sf e}={\sf c}$ it suffices to observe that $\#\kappa\not\in g_0(J)$ for $J\in X$ and that $g(J):=g_0(J)\cup \{(l\restr k)^*\,|\,k\in\mathbb{N}\}$ is consistent. By transitivity of $\Phi$ it follows that $g\in\Phi_{\sf c}(f)$,  which completes the proof for the case ${\sf e}={\sf c}$.

Let ${\sf e}\in\{{\sf K3},{\sf N3},{\sf Nve}\}$. As discussed in Remark \ref{Transp} there is a valuation function $g\in\Phi_{\sf e}(g_0)$ such that $\kappa\in g_0(J')$. By the properties of $\Phi_{\sf e}$ this implies that $\{(l\restr k)^*\,|\,k\in\mathbb{N}\}\subseteq g(J')$. By transitivity of $\Phi_e$ this establishes (Claim) and concludes the proof.

\end{document}